\documentclass[final]{amsart}
\copyrightinfo{}{}
\usepackage{amssymb}
\usepackage{xcolor}
\definecolor{darkblue}{rgb}{0,0,.8}
\definecolor{darkred}{rgb}{0.8,0,0}
\usepackage{hyperref}
\hypersetup{
  pdftex=true,
  colorlinks=true,
  allcolors = darkred,
  linkcolor=darkblue, 
  menucolor=darkblue,
  urlcolor = darkblue,
  frenchlinks=false,
  pdfborder={0 0 0},
  naturalnames=false,
  hypertexnames=false,
  breaklinks
}

\usepackage[capitalize]{cleveref}
\crefname{section}{section}{sections}

\Crefname{figure}{Figure}{Figures}

\crefformat{equation}{\textup{#2(#1)#3}}
\crefrangeformat{equation}{\textup{#3(#1)#4--#5(#2)#6}}
\crefmultiformat{equation}{\textup{#2(#1)#3}}{--\textup{#2(#1)#3}}
{, \textup{#2(#1)#3}}{, and \textup{#2(#1)#3}}
\crefrangemultiformat{equation}{\textup{#3(#1)#4--#5(#2)#6}}%
{ and \textup{#3(#1)#4--#5(#2)#6}}{, \textup{#3(#1)#4--#5(#2)#6}}{, and \textup{#3(#1)#4--#5(#2)#6}}

\Crefformat{equation}{#2Equation~\textup{(#1)}#3}
\Crefrangeformat{equation}{Equations~\textup{#3(#1)#4--#5(#2)#6}}
\Crefmultiformat{equation}{Equations~\textup{#2(#1)#3}}{ and \textup{#2(#1)#3}}
{, \textup{#2(#1)#3}}{, and \textup{#2(#1)#3}}
\Crefrangemultiformat{equation}{Equations~\textup{#3(#1)#4--#5(#2)#6}}%
{ and \textup{#3(#1)#4--#5(#2)#6}}{, \textup{#3(#1)#4--#5(#2)#6}}{, and \textup{#3(#1)#4--#5(#2)#6}}

\crefdefaultlabelformat{#2\textup{#1}#3}

\usepackage{graphicx}
\usepackage{epstopdf}
\usepackage{psfrag}
\usepackage[caption=false]{subfig}
\usepackage{pgfplots}
\usepackage{pgfplotstable}
\usetikzlibrary{calc}
\usetikzlibrary{plotmarks}
\usetikzlibrary{arrows.meta}
\usepgfplotslibrary{patchplots}
\pgfplotsset{plot coordinates/math parser=false}
\newlength\figureheight
\newlength\figurewidth

\newcommand{\logLogSlopeTriangle}[5]
{

    \pgfplotsextra
    {
        \pgfkeysgetvalue{/pgfplots/xmin}{\xmin}
        \pgfkeysgetvalue{/pgfplots/xmax}{\xmax}
        \pgfkeysgetvalue{/pgfplots/ymin}{\ymin}
        \pgfkeysgetvalue{/pgfplots/ymax}{\ymax}

        \pgfmathsetmacro{\xArel}{#1}
        \pgfmathsetmacro{\yArel}{#3}
        \pgfmathsetmacro{\xBrel}{#1-#2}
        \pgfmathsetmacro{\yBrel}{\yArel}
        \pgfmathsetmacro{\xCrel}{\xArel}

        \pgfmathsetmacro{\lnxB}{\xmin*(1-(#1-#2))+\xmax*(#1-#2)} 
        \pgfmathsetmacro{\lnxA}{\xmin*(1-#1)+\xmax*#1} 
        \pgfmathsetmacro{\lnyA}{\ymin*(1-#3)+\ymax*#3} 
        \pgfmathsetmacro{\lnyC}{\lnyA+#4*(\lnxA-\lnxB)}
        \pgfmathsetmacro{\yCrel}{\lnyC-\ymin)/(\ymax-\ymin)} 

        \coordinate (A) at (rel axis cs:\xArel,\yArel);
        \coordinate (B) at (rel axis cs:\xBrel,\yBrel);
        \coordinate (C) at (rel axis cs:\xCrel,\yCrel);

        \draw[#5]   (A)-- node[pos=0.5,anchor=south] {1}
                    (B)-- 
                    (C)-- node[pos=0.5,anchor=west] {#4}
                    cycle;
    }
}
\newtheorem{theorem}{Theorem}
\newtheorem{lemma}[theorem]{Lemma}

\theoremstyle{definition}

\newtheorem{corollary}[theorem]{Corollary}

\theoremstyle{remark}
\newtheorem{remark}[theorem]{Remark}
\newtheorem{problem}[theorem]{Problem}

\definecolor{darkgreen}{rgb}{0,0.5,0}

\newcommand{\RR}{\mathbb{R}} 
\newcommand{\NN}{\mathbb{N}} 

\def\T{\mathcal{T}}  
\def\E{\mathcal{E}}  
\def\N{\mathcal{N}}  
\def\O{\mathcal{O}}  
\def\P{\mathcal{P}}  
\def\C{\mathcal{C}}  

\def\Er{\E_{\Gamma}}
\def\Et{\E_K} 

\def\normal{\mathbf{n}}

\def\V{\mathcal{V}} 
\def\K{\mathcal{K}} 
\def\S{\mathcal{S}} 

\def\AA{\mathcal{A}}
\def\BB{\mathcal{B}}
\def\HH{\mathcal{H}}

\DeclareMathOperator{\sgn}{sign}
\def\diam{{\operatorname{diam}}}
\def\div{\operatorname{div}}

\def\dt{\partial_t}

\def\dn{\partial_{\normal}}
\def\dtau{d_\tau}

\def\A{\mathbf{A}}
\def\b{\mathbf{b}}
\def\c{c}

\def\u{\mathbf{u}}
\def\v{\mathbf{v}}
\def\w{\mathbf{w}}

\def\II{\mathcal{I}} 
\def\IIh{\II_h} 

\def\set#1#2{\left\{#1\,:\,#2\right\}}

\newcommand{\norm}[3][]{#1\|#2#1\|_{#3}}

\newcommand{\dual}[3]{\langle#1\hspace*{.5mm},#2\rangle_{#3}}
\newcommand{\product}[3]{(#1\hspace*{.5mm},#2)_{#3}}
\newcommand{\intd}[1]{\,\mathrm{d}#1}

\begin{document}

\title[Non-symmetric FVM-BEM coupling]{Stable non-symmetric coupling of the finite volume and the boundary element method 
		  for convection-dominated parabolic-elliptic interface problems}

\author[C. Erath]{Christoph Erath}
\address{TU Darmstadt, 
  Department of Mathematics, Dolivostra\ss{}e 15, 64293 Darmstadt, Germany}
\email{erath@mathematik.tu-darmstadt.de}

\author[R. Schorr]{Robert Schorr}
\address{TU Darmstadt, Graduate School of Computational Engineering, Dolivostra\ss{}e 15, 64293 Darmstadt, Germany}
\email{schorr@gsc.tu-darmstadt.de}
\thanks{\textbf{Funding:} The research of the second author was supported by the \emph{Excellence Initiative} 
  of the German Federal and State Governments 
  and the \emph{Graduate School of Computational Engineering} at TU Darmstadt.}

\subjclass[2010]{65N08, 65N38, 65N40, 65N12, 65N15, 82B24}

\keywords{parabolic-elliptic interface problem, convection-dominated,
finite volume method, upwind stabilization, boundary element method, non-symmetric coupling,
 method of lines, backward Euler,
 convergence, a~priori error estimates}

\date{14.05.2018}

\begin{abstract}
Many problems in electrical engineering or fluid mechanics can be 
modeled by parabolic-elliptic interface problems,
where the domain for the exterior elliptic problem might be unbounded.
A possibility to solve this class of problems numerically is 
the non-symmetric coupling of finite elements (FEM) and boundary elements (BEM) 
analyzed in~\cite{Egger:2017}.
If, for example, the interior problem represents a fluid, 
this method is not appropriate
since FEM in general lacks conservation of numerical fluxes and in case of
convection dominance also stability.

A possible remedy to guarantee both is the use
of the vertex-centered finite volume method (FVM) with an 
upwind stabilization option.
Thus we propose a (non-symmetric) coupling of FVM and BEM for a semi-discretization of the
underlying problem. For the subsequent time discretization we introduce two
options: a variant
of the backward Euler method which allows us to develop an analysis under minimal regularity assumptions 
and the classical backward Euler method.
We analyze both, the semi-discrete and the fully discrete system, in terms of convergence 
and error estimates. Some numerical examples illustrate the theoretical findings and
give some ideas for practical applications.
\end{abstract}

\maketitle

\section{Introduction} \label{sec:intro}
We consider a parabolic-elliptic interface problem on a bounded domain
$\Omega \subset \RR^2$ with $\diam(\Omega)<1$
and its complement $\Omega_e = \RR^2 \setminus \Omega$. 
The domains are connected through a polygonal Lipschitz boundary 
$\Gamma=\partial \Omega= \partial\Omega_e$.
An extension of our analysis to three dimensions is straightforward.
Note that the assumption $\diam(\Omega)<1$ is needed in two dimensions 
to ensure ellipticity of the single layer operator defined below.
This can always be achieved by scaling.

The known \emph{model parameters} are a symmetric
diffusion matrix $\A$, a possibly dominating velocity field $\b$,
and a reaction coefficient $\c$.
Furthermore,  the coupling boundary $\Gamma$ is
divided in an inflow and outflow part, 
namely $\Gamma^{in}:=\set{x\in\Gamma}{\b(x)\cdot\normal(x)<0}$ and
$\Gamma^{out}:=\set{x\in\Gamma}{\b(x)\cdot\normal(x)\geq 0}$, respectively, 
where $\normal$ is the normal vector on $\Gamma$ pointing outwards with respect
to $\Omega$. 
Then our model problem reads:
Find $u$ and $u_e$ such that 
\begin{alignat}{2}
\partial_t u + \div (-\A \nabla u + \b u)+\c u  &= f              
&\qquad& \text{in }\Omega\times (0,T),   \label{eq:model1} \\
           -\Delta u_e &= 0              
&\qquad& \text{in } \Omega_e \times (0,T), \label{eq:model2} \\
\intertext{with coupling conditions across the interface given by}
                     u &= u_e + g_1         
&\qquad& \text{on } \Gamma \times (0,T), \label{eq:model3} \\
( \A \nabla u-\b u)\cdot\normal &= \partial_{\normal} u_e +g_2\ 
&\qquad& \text{on } \Gamma^{in} \times (0,T), \label{eq:model4.1} \\
( \A \nabla u)\cdot\normal &= \partial_{\normal} u_e +g_2\ 
&\qquad& \text{on } \Gamma^{out} \times (0,T), \label{eq:model4.2} 
\intertext{with a fixed time $T>0$. 
To ensure the uniqueness of the solution, we additionally require the following initial and radiation conditions}
            u(\cdot,0) &= q               &\qquad& \text{on } \Omega, \label{eq:model5}\\
            u_e(x,t) &= a(t) \log|x| + \O(|x|^{-1})  &\qquad& |x| \to \infty. \label{eq:model6}
\end{alignat}
The function $a(t):[0,T]\to\RR$ is unknown but can be computed from the solution, 
see~\cref{rem:repformula}.   
The \emph{model input data} are $q$, $f$, $g_1$, and $g_2$. 
Note that the interior problem is the time dependent prototype of  
transport and flow of a substance in a porous medium coupled to 
a diffusion process in an unbounded domain.
The coupling to the exterior problem can  also be seen as a ``replacement'' of (maybe) unknown
Dirichlet and/or Neumann data, see also~\cite[Remark 2.1]{Erath:2012-1}.
For a model problem in three dimensions, we only have to replace the radiation
condition~\cref{eq:model6} by $u_e(x,t)=\O(|x|^{-1})$, $|x| \to \infty$.

The recent work~\cite{Egger:2017} analyzes
the numerical approximation of a parabolic-elliptic interface problem
by a non-symmetric coupling of the finite element
method (FEM) and the boundary element method (BEM)
followed by a variant of the backward Euler method for the discretization in time.
This allows us to state quasi-optimality results in the natural energy norm for both, 
the semi discrete system and the fully
discrete system under minimal regularity assumptions on the data and the solution.
Although~\cite{Egger:2017} provides an analysis 
of the discrete system 
only for the simple model problem
$\A=\mathbf{I}$, $\b=(0,0)^T$, and $\c=0$, the arguments can be easily applied to the
more general model problem~\cref{eq:model1}--\cref{eq:model6}. 
This class of problems includes convection-diffusion-reaction equations in
the interior domain which can be dominated by convection and thus pose some
challenges to the numerical method. In the convection dominated case, the FEM-BEM coupling is
not stable anymore and yields unwanted oscillations. 
In a study of different stable discretization methods for 
convection-diffusion equations with dominating convection 
the work~\cite{Augustin:2001} concludes 
that the Streamline Upwind Petrov Galerkin (SUPG) method or
the finite volume method (FVM) with upwind stabilization are the simplest approaches 
and often sufficient. 
Note that SUPG creates sharper layers than FVM with upwinding 
but does not completely avoid spurious 
oscillations. Furthermore, the numerical fluxes are not conservative.
The FVM, however,  provides a natural upwind stabilization for convection dominated
problems and avoids spurious oscillations. Additionally, it
preserves conservation of numerical fluxes due to an approximation
of the balance equation and on certain grids it fulfills the maximum principle. 
Therefore, FVM is often the method of choice for fluid mechanics applications.
Our second ingredient is the BEM, which is based on an integral equation formulation
with the fundamental solution of the differential operator
to represent the solution of the exterior problem by the Cauchy data
$(u_e,\partial_{\normal} u_e)|_\Gamma$ on the boundary $\Gamma$, see, e.g.,~\cite{McLean:2000-book}. 
The discretization problem is then reduced to its boundary.
Finally, the solution
can be post processed through a representation formula in the domain 
and the fluxes are in a sense locally conservative.
This strategy avoids the truncation of an unbounded domain which would be
necessary for domain based methods like FEM or FVM.

This motivates us to consider the coupling of the vertex-centered finite volume method 
(FVM) with BEM. In~\cite[vertex-centered FVM-BEM]{Erath:2012-1,Erath:2017-1} 
and~\cite[cell-centered FVM-BEM]{Erath:2013-2} this coupling combination
was considered as well, but only for stationary interface problems. 
In the literature there exist different coupling strategies with BEM.
The easiest one is the so-called non-symmetric
coupling approach~\cite{Johnson:1980-1,MacCamy:1987} which will be considered in this work.
It is well known that implicit methods for parabolic problems are preferable
over explicit methods. 
Hence we will either use a variant of the backward Euler or
the classical backward Euler method  in the time regime. 
These two time discretizations only differ
in the right-hand side. The variant is 
computational more expensive but allows us to state quasi-optimal results under
minimal regularity assumptions also in the time component 
of the solution for the fully discrete system. 
For the classical Euler scheme, however, we need standard regularities due to Taylor expansion
techniques.
In contrast to the analysis of the FEM-BEM coupling, we are not able to 
achieve the analysis for our
coupling method
in the full energy norm, i.e., we have to omit the dual norm of the time derivative.
The finite volume formulation does not allow a right-hand side that is less regular than $L^2$
and therefore does not permit the tricks to handle the dual norm, see, e.g.,
the proof of~\cite[Lemma 8]{Egger:2017}. 
In contrast to a standard FEM-BEM coupling our FVM-BEM coupling does not have a ``global'' Galerkin orthogonality. Hence
we will also have to handle some extra terms concerning the model input data. However, 
the analysis still holds for minimal
regularity requirements on the solution.

We summarize our main results as follows: 
\begin{itemize}
\item We formulate the non-symmetric coupling of the finite 
volume method with the boundary element method which 
leads to the semi-discretization of the model problem. 
\item We show convergence of the semi-discrete scheme under minimal regularity requirements on the solution 
      and provide error estimates with optimal rates.
\item For the full discretization with the variant of the backward Euler scheme we 
      provide convergence under minimal regularity assumptions on the solution and provide
      error estimates with optimal rates. If we use the classical Euler scheme for time discretization
      the usual regularity assumptions for the time component lead to first order error estimates.
\item It is important to note that the analysis still holds if we use an upwind stabilization or if we consider the model problem in three dimensions.
\item We can apply the analysis in this work also for standalone FVM, 
      i.e., one has Dirichlet and/or
      inflow/outflow Neumann boundary conditions on $\Gamma$ instead of the coupling conditions. 
      Note that our results also improve results in the 
      literature, e.g.,~\cite{Ewing:2002,Chatzipantelidis:2004}.
\end{itemize}

The rest of the paper is organized as follows:
In~\cref{sec:prelim} we state the basic notation, introduce the triangulation and
discrete spaces, state a variational formulation of our PDE-system and the well-posedness of the
model problem.
In~\cref{sec:fvm} the finite volume method and the upwind stabilization are introduced.
\Cref{sec:semi-disc} defines the semi-discretization 
of the whole model problem and analyzes convergence of this discretization to the continuous solution with rates.
\Cref{sec:fully-disc} states convergence and a~priori results for the full discretization
with both time discretizations.
Lastly we provide some numerical experiments in~\cref{sec:numerics}
to support the preceding theoretical results and state some concluding 
remarks in~\cref{sec:conclusions}.

\section{Assumptions and weak coupling formulation} \label{sec:prelim}

In this section, we first introduce some basic notation and assumptions.
Then we formulate and analyze a weak formulation of our model problem. 
Throughout, $C>0$ denotes a constant which may vary at different occurrences.
Furthermore, we abbreviate the relation $a\leq C b$ by $a\lesssim b$.

\subsection{Notation and basic assumptions}
We write $L^2(\cdot)$ and $H^s(\cdot)$, $s\in\RR$,
for the usual Lebesgue- or Sobolev spaces.  
The space of all traces of functions from $H^s(\Omega)$
is $H^{s-1/2}(\Gamma)$, $s>1/2$, see~\cite{Evans:2010-book,McLean:2000-book} for details.
We denote the $L^2$ scalar product for $\omega\subset\Omega$ by $\product{\cdot}{\cdot}{\omega}$
and duality between $H^{s}(\Gamma)$ and $H^{-s}(\Gamma)$ is given by the extended 
$L^2$-scalar product $\dual{\cdot}{\cdot}{\Gamma}$. 

To shorten the notation, we will use
\begin{align*}
H&=H^1(\Omega) \qquad \text{and} \qquad B=H^{-1/2}(\Gamma)
\end{align*}
for the main function spaces which are natural to this problem.
Furthermore, we denote by
\begin{align*}
H_T=L^2(0,T;H) \qquad \text{and} \qquad B_T=L^2(0,T;B)
\end{align*}
the corresponding Bochner spaces of functions on $[0,T]$ with values in $H$ and $B$, respectively. 
The associated dual spaces are given by $H'=H^1(\Omega)'$ and $B'=H^{1/2}(\Gamma)$ 
as well as $H_T'=L^2(0,T;H')$ and $B_T'=L^2(0,T;B')$. 
We also abbreviate the spaces $L^2(0,T;L^2(\Omega))$ and $L^2(0,T;L^2(\Gamma))$ by
\begin{align*}
L_{T,\Omega}^2:=L^2(0,T;L^2(\Omega))\qquad \text{and} \qquad L_{T,\Gamma}^2=L^2(0,T;L^2(\Gamma)),
\end{align*}
respectively.
All spaces above are Hilbert spaces if equipped 
with their natural norms, e.g., $\norm{v}{H_T}^2 = \int_0^T \norm{v(t)}{H}^2\, \intd t$. We further use
\begin{align*}
Q_T = \set{v \in H_T}{\dt v \in H_T' \text{ and } v(0)=q}
\end{align*}
to denote the natural energy space for the parabolic problem on $\Omega$.
To simplify notation we also use a product space and norm notation, e.g., we equip the space
$\HH:=H\times B=H^1(\Omega)\times H^{-1/2}(\Gamma)$ with the norm
\begin{align*}
 \norm{\v}{\HH}^2:=\norm{v}{H^1(\Omega)}^2+\norm{\psi}{H^{-1/2}(\Gamma)}^2
\end{align*}
for $\v=(v,\psi)\in\HH$.

For the \emph{model parameters} we assume the following regularities:
The diffusion matrix $\A:\Omega\to\RR^{2\times 2}$ has piecewise 
Lipschitz continuous entries; i.e., entries in 
$W^{1,\infty}(K)$ for every $K\in\T$, where $\T$ is a mesh of $\Omega$ introduced below in 
\cref{subsec:triangulation}. Additionally, $\A$ is bounded, symmetric, 
and uniformly positive definite. 
The minimum eigenvalue of $\A$ is $\lambda_{\min}(\A)$.
Furthermore, $\b\in W^{1,\infty}(\Omega)^2$ and  $\c\in L^{\infty}(\Omega)$ fulfill
\begin{align*}
  \frac{1}{2}\div\b(x)+\c(x)> 0\quad \text{for almost every }x\in\Omega.
\end{align*}
Hence the system is indeed parabolic-elliptic. 
For the \emph{model input data} we allow $q\in L^2(\Omega)$,
$f\in H'_T$,  $g_1\in B'_T$
and $g_2 \in B_T$.

\begin{remark}
To handle the case $\frac{1}{2}\div \b + \c = 0$ we use a standard transformation
 of the whole system, i.e., multiplying by $e^{-\lambda t}$ with $\lambda>0$
 leads to a system~\cref{eq:model1}--\cref{eq:model6}
 in the variables $u_{\lambda}=u e^{-\lambda t}$ and $u_{e,\lambda}=u_e e^{-\lambda t}$
 with an additional factor $\lambda u_\lambda$ in \cref{eq:model1}.
 Hence we fulfill $\frac{1}{2}\div \b + \c +\lambda > 0$ which is in 
 fact the above situation.
\end{remark}

\begin{remark}
 Note that the model parameters $\A$, $\b$, and
 $\c$ are time-independent. 
 The fully discrete analysis with the classical Euler scheme for time discretization 
 can be easily transferred to time-dependent parameters.
 For the variant version the extension is an open question.
\end{remark}

\subsection{Variational formulation}

A weak formulation of the model problem~\crefrange{eq:model1}{eq:model6}
can be derived with a non-symmetric
coupling approach with the boundary integral operators $\V$ and $\K$.
The derivation of the weak formulation in~\cite{Egger:2017}
applies for our more general problem~\crefrange{eq:model1}{eq:model6}
by some obvious modifications and is thus skipped. 

\begin{problem}[Variational problem] \label{prob:variational}
Given $f \in H'_T$, $g_1 \in B_T'$, and $g_2 \in B_T$, 
find $u \in Q_T$ and $\phi \in B_T$ such that 
\begin{align}
\label{eq:vp1}
\dual{\dt u(t)}{v}{\Omega} +  \AA(u(t),v) - \dual{\phi(t)}{v}{\Gamma} 
&= \dual{f(t)}{v}{\Omega} + \dual{g_2(t)}{v}{\Gamma},\\
\label{eq:vp2}
\dual{(1/2-\K) u(t)|_\Gamma}{\psi}{\Gamma} + \dual{\V \phi(t)}{\psi}{\Gamma} 
&= \dual{(1/2 - \K) g_1(t)}{\psi}{\Gamma}  
\end{align}
for all test functions $v \in H=H^1(\Omega)$ and $\psi \in B=H^{-1/2}(\Gamma)$ 
and for a.e. $t \in [0,T]$.
The bilinear form $\AA(\cdot,\cdot)$ is defined by
\begin{align*}
  \AA(u(t),v):=\product{\A\nabla u(t)-\b u(t)}{\nabla v}{\Omega}+\product{\c u(t)}{v}{\Omega}
+\dual{\b\cdot\normal \,u(t)}{v}{\Gamma^{out}}.
\end{align*}
\end{problem}
The exterior formulation, i.e., the transformation of the exterior problem~\cref{eq:model2} 
and~\cref{eq:model6} into an integral equation, uses the single layer operator $\V$ and
the double layer operator $\K$. For smooth
enough input and $x\in \Gamma$ they are given by
\begin{align*}
(\V \psi)(x)=\int_{\Gamma}\psi(y)G(x-y)\,\intd{s}_y
 	\quad\text{and}\quad
(\K \theta)(x)=
  \int_{\Gamma}\theta(y)\frac{\partial}{\partial 
  \normal_{y}}G(x-y)\,\intd{s}_y,
\end{align*}
where $\normal_y$ is a normal vector with respect to $y$ and
$G(z)=-\frac{1}{2\pi}\log|z|$ is the fundamental solution for the Laplace operator.
As stated in~\cite[Theorem 1]{Costabel:1988-1}, these operators can be extended to linear 
bounded operators
\begin{align*}
\V \in L\big(H^{s-1/2}(\Gamma),H^{s+1/2}(\Gamma)\big),
	\qquad
\K \in L\big(H^{s+1/2}(\Gamma),H^{s+1/2}(\Gamma)\big),
	\quad s\in [-\tfrac12,\tfrac12].  
\end{align*}
The double layer operator $\K$ fulfills a contraction property with constant
$C_{\K}\in[1/2,1)$.
Furthermore, $\V$ is symmetric and due to the assumption $\diam(\Omega)<1$ also 
$H^{-1/2}(\Gamma)$ elliptic.
Therefore
\begin{align*}
\norm{\cdot}{\V}^2:=\dual{\V\cdot}{\cdot}{\Gamma}
\end{align*}
defines a norm in $H^{-1/2}(\Gamma)$ which is equivalent to $\norm{\cdot}{H^{-1/2}(\Gamma)}$.

For convenience we write 
the system~\cref{eq:vp1}--\cref{eq:vp2} in a more compact form.
With the product spaces $\HH=H \times B$ and $\HH_T=Q_T \times B_T$ we introduce 
the continuous bilinear form
$\BB:\HH_T\times\HH \to \RR$ by
\begin{align}
 \label{eq:B}
 \begin{split}
 \BB((u(t),\phi(t));(v,\psi)) := &  \AA(u(t),v) - \dual{\phi(t)}{v}{\Gamma} \\
 &{}+ \dual{(1/2-\K) u(t)|_\Gamma}{\psi}{\Gamma} + \dual{\V \phi(t)}{\psi}{\Gamma},
 \end{split}
\end{align}
and the linear functional $F:\HH\to\RR$ by
\begin{align}
 \label{eq:F}
 F((v,\psi);t):=  \dual{f(t)}{v}{\Omega} 
 + \dual{g_2(t)}{v}{\Gamma} + \dual{(1/2 - \K) g_1(t)}{\psi}{\Gamma}.
\end{align}
Then~\cref{eq:vp1,eq:vp2} is equivalent to: 
\begin{problem}
Find
$\u=(u,\phi)\in\HH_T$ such that
\begin{align}
	\label{eq:weakAddedForm}
	\dual{\dt u(t)}{v}{\Omega} +\BB(\u(t);\v)=F(\v;t) \quad \text{for all } \v=(v,\psi)\in \HH \text{ and a.e. } t\in [0,T].
\end{align}
\end{problem}
\begin{remark}
 \label{rem:repformula}
 The complete Cauchy data for the exterior problem are
 $(u_e(t)|_\Gamma,\dn u_e(t)|_\Gamma)=(u(t)|_\Gamma-g_1,\phi)$, which
 results from the solution of the system~\cref{eq:weakAddedForm}.
 Then the solution to the exterior problem can be expressed by the
 representation formula~\cite{McLean:2000-book}, i.e.,
 \begin{align*}
 u_e(x,t) = \int_{\Gamma} \partial_{\normal_y} G(x,y) u_e(y)|_\Gamma\,\intd{s}_y 
 - \int_{\Gamma} G(x,y) \dn u_e(y)|_\Gamma\,\intd{s}_y.
 \end{align*}
 The factor $a(t)$ from~\cref{eq:model6} is calculated
 from $a(t)=\frac{1}{2\pi}\int_\Gamma \phi\intd{s}$. 
\end{remark}
\begin{theorem}[Well-posedness of the model problem]
Let $\lambda_{\min}(\A)-\frac{1}{4}C_\K>0$ with $C_{\K}\in[1/2,1)$.
The weak solution $\u=(u,\phi)\in \HH_T=Q_T\times B_T$
of the model problem~\cref{eq:vp1}--\cref{eq:vp2}
or~\cref{eq:weakAddedForm} exists and is unique. Furthermore, there holds
\begin{align*}
\norm{u}{H_T}+\norm{\phi}{B_T} + \norm{\dt u}{H_T'} 
\leq C\big(\norm{f}{H_T'}+\norm{q}{L^2(\Omega)}+\norm{g_2}{B_T}+\norm{g_1}{B_T'}\big) 
\end{align*}
with a constant $C>0$ which depends only on the domain $\Omega$ 
and the time horizon $T$.
\end{theorem}
\begin{proof}
 The bilinear form $\BB(\cdot;\cdot)$ is $\HH$-elliptic and continuous, 
 see~\cite[Theorem 1 and Remark 2]{Erath:2017-1}.
 Hence, the proof of~\cite[Theorem 4]{Egger:2017} is applicable.
\end{proof}
\begin{remark}
 The condition $\lambda_{\min}(\A)-C_\K/4>0$ results from our non-symmetric coupling
 approach. In general, this is not necessary for well-posedness of the model
 problem~\cref{eq:model1}--\cref{eq:model6}, e.g., if one uses the symmetric 
 coupling approach.
\end{remark}
\section{Vertex-centered Finite Volume Method}\label{sec:fvm}
In contrast to the previous work~\cite{Egger:2017} 
we employ FVM instead of FEM to solve the problem in the interior domain.
Since FVM is based on a balance equation, it naturally conserves numerical fluxes.
Furthermore, an (optional) upwinding strategy guarantees stability of the numerical 
scheme also for convection dominated problems but with retention of numerical flux conservation.
An early (if not first) mathematical analysis of the vertex-centered FVM 
is found in~\cite{Bank:1987} and~\cite{Hackbusch:1989}.
Later works put the method into a more modern framework, 
see, e.g.,~\cite{Ewing:2002} or~\cite{Chatzipantelidis:2004} for parabolic problems,
or a C\'ea-type estimate for general second order
elliptic PDE in~\cite{Erath:2016-1,Erath:2017-2}.
Since the FVM is based on two meshes we have to introduce some additional notation.
From now on we assume some more regularity for the input data, namely
$f\in L_{T,\Omega}^2$ and $g_2 \in L_{T,\Gamma}^2$.

\subsection{Triangulation and discrete spaces}
\label{subsec:triangulation}

\subsection*{Primal mesh}
Let $\T$ denote a triangulation or primal mesh of $\Omega$ 
consisting of non-degenerate closed triangles
denoted by $K\in\T$. 
The corresponding sets of nodes and edges are 
denoted by $\N$ and $\E$, respectively.
We write $h_K:=\sup_{x,y\in K}|x-y|$ for the Euclidean diameter of $K\in\T$
and $h_E$ for the length of an edge  $E\in\E$. The maximum mesh size
is $h:=\max_{K\in\T}h_K$.
The triangulation is shape regular, i.e., $\T$ is regular in the sense
of Ciarlet~\cite{Ciarlet:1978-book} and the ratio of the diameter $h_K$ 
of any element $K\in\T$ to the diameter
of its largest inscribed ball is bounded by a constant independent
of $h_K$, the so called shape regularity constant. 
Furthermore, we denote by $\Et\subset\E$ the set of all edges of $K$, i.e.,
$\Et:=\set{E\in\E}{E\subset \partial K}$ and by
$\Er:=\set{E\in\E}{E\subset\Gamma}$ the set of all edges on the boundary $\Gamma$.

\subsection*{Dual mesh}

For a visual construction of the dual mesh $\T^*$ from the primal mesh $\T$ we
refer to~\cite[Figure 1]{Erath:2017-1}.
We build boxes, called control volumes, by connecting the center of gravity
of an element $K\in\T$ with the midpoint of the edges $E\in\Et$. 
These control volumes constitute a new triangulation $\T^*$ of $\Omega$
whose elements are non-degenerate and closed because
of the non-degeneracy of the elements of the primal mesh $\T$.
For every vertex  $a_i\in\N$ of $\T$ ($i=1\ldots \#\N$)
we associate a unique box $V_i\in \T^*$ containing $a_i$. 

\subsection*{Discrete spaces and piecewise constant interpolation}

To define the FVM-BEM coupling for the spatial discretization
we introduce the discrete spaces
\begin{align*}
 \S^1(\T)&:=\set{v\in\C(\Omega)}{v|_K \text{ affine for all } K\in\T},\\
 \P^0(\Er)&:=\set{v\in L^2(\Gamma)}{v|_E \text{ constant on } E\in\Er},\\
 \P^0(\T^*)&:=\set{v\in L^2(\Omega)}{v|_V \text{ constant on } V\in\T^*}.
\end{align*}
By means of the characteristic function $\chi_i^*$ over the volume $V_i$ associated with $a_i\in\N$
we write $v_h^*\in\P^0(\T^*)$ as
\begin{align*}
   v_h^*=\sum_{a_i\in\N}v_i^*\chi_i^*,
 \end{align*}
with $v_i^*\in\RR$. In that sense we define 
the $\T^*$-piecewise constant interpolation operator
\begin{align*}
   \IIh^*:\C(\overline\Omega)\to\P^0(\T^*),\quad
   (\IIh^*v)(x):=\sum_{a_i\in\N} v(a_i)\chi_i^*(x)
\end{align*}
which has the following properties:
\begin{lemma}
Let $K\in\T$ and $E\in\Et$. For $v_h\in \S^1(\T)$ there holds
\begin{align}
 \label{eq:piecewiseintF}
 \int_E (v_h-\IIh^* v_h) \, \intd{s} & = 0,\\
 \label{eq:piecewisepropT}
\norm{v_h - \IIh^* v_h}{L^2(K)} &\leq h_K\norm{\nabla v_h}{L^2(K)}, \\
\label{eq:piecewisepropF}
\norm{v_h - \IIh^* v_h}{L^2(E)} &\leq C h_E^{1/2}\norm{\nabla v_h}{L^2(K)},\\
\label{eq:piecewiseStability}
\norm{\IIh^* v_h}{L^2(\Omega)}  &\leq C \norm{v_h}{L^2(\Omega)}.
\end{align}
The constant $C>0$ depends only on the shape regularity constant.
\end{lemma}
\begin{proof}
The estimates~\cref{eq:piecewiseintF}--\cref{eq:piecewisepropF} are well known, 
see e.g.~\cite[Lemma 3]{Erath:2017-1}. 
The stability~\cref{eq:piecewiseStability} follows 
from~\cref{eq:piecewisepropT} 
and an inverse inequality~\cite[Theorem 3.2.6]{Ciarlet:1978-book}, i.e.,
\begin{align*}
 \norm{\IIh^* v_h}{L^2(\Omega)} &\leq \norm{\IIh^* v_h - v_h}{L^2(\Omega)} + \norm{v_h}{L^2(\Omega)} \\
 &\leq Ch \norm{\nabla v_h}{L^2(\Omega)} + \norm{v_h}{L^2(\Omega)} \leq C \norm{v_h}{L^2(\Omega)}.
\end{align*}
\end{proof}
\begin{lemma}[{\cite[Lemma 2.2]{Chou:1999}}]\label{lem:iih_norm}
The operator $\IIh^*$ is self-adjoint in the $L^2$ scalar product, 
which means that for all $v_h,w_h \in \S^1(\T)$
\begin{align}
 \label{eq:selfadjoint}
\product{w_h}{\IIh^* v_h}{\Omega} = \product{v_h}{\IIh^* w_h}{\Omega}.
\end{align}
This allows us to define the norm
\begin{align}
\norm{w_h}{\chi}:=\product{w_h}{\IIh^* w_h}{\Omega}^{1/2}, \label{eq:piecewiseeqNorm}
\end{align}
which is equivalent to $\norm{w_h}{L^2(\Omega)}$.
\end{lemma}

\subsection{Finite volume bilinear form}
In the following we omit the dependence on $t$ in the notation.
All expressions hold for a.e. $t\in [0,T]$.
A finite volume method is based on the reformulation of the differential 
equation as a conservation law, i.e.,
a balance equation through the boundary of some cells. 
We achieve that if we formally integrate our interior
equation~\cref{eq:model1} over the control volumes $V\in\T^*$ 
and use the Gaussian divergence theorem 
to rewrite it;

\begin{align*}
\int_V f \intd{x}
= \int_V \partial_t u\intd{x}&+\int_{\partial V \setminus \Gamma} (-\A\nabla u + \b u) \cdot \normal \intd{s}\\
&+ \int_{\partial V \cap \Gamma} -(\A\nabla u - \b u) \cdot \normal \intd{s} 
+ \int_V \c u \intd{x}.
\end{align*}
Now we make use of the jump 
relations~\crefrange{eq:model4.1}{eq:model4.2} on the boundary.
If we additionally replace $u$ by $u_h\in \S^1(\T)$ 
and $\phi=\dn u_e|_\Gamma$ by $\phi_h\in \P^0(\Er)$ we get
\begin{align*}
\int_V \partial_t u\intd{x}
+\int_{\partial V \setminus \Gamma} (-\A\nabla u_h &+ \b u_h) \cdot \normal \intd{s} 
+ \int_{\partial V \cap \Gamma^\text{out}}\b \cdot \normal u_h \intd{s}\\
 &- \int_{\partial V \cap \Gamma}\phi_h \intd{s} + \int_V \c u_h \intd{x} 
 =\int_V f \intd{x} + \int_{\partial V \cap \Gamma} g_2 \intd{s} 
\end{align*}
for all $V\in\T^*$.
By testing the equation with a piecewise constant function on the dual mesh $\T^*$, 
we write the system as a Petrov-Galerkin method.
Indeed, with $\IIh^*v_h\in \P^0(\T^*)$ for all $v_h\in\S^1(\T)$ the FVM reads

\begin{align}\label{eq:fvmformulation}
\AA_V(u_h,v_h)-\dual{\phi_h}{\IIh^*v_h}{\Gamma} 
= \product{f}{\IIh^*v_h}{\Omega} + \dual{g_2}{\IIh^*v_h}{\Gamma}
\end{align}
with the finite volume bilinear form 
$\AA_V \colon \S^1(\T)\times\S^1(\T) \to \RR$ defined by
\begin{align}
\begin{split} \label{eq:fvmbilinearform}
 \AA_V(u_h,v_h) := \sum_{a_i\in\N} v_h(a_i)  &\left( \int_{\partial V_i \setminus \Gamma} (-\A\nabla u_h + \b u_h) \cdot \normal \intd{s} \right. \\
& \left. {}+  \int_{\partial V_i \cap \Gamma^\text{out}}\b \cdot \normal u_h \intd{s} + \int_{V_i} \c u_h \intd{x} \right).
\end{split}
\end{align}

\begin{remark}
Under certain conditions, e.g., $\A$ is only $\T$-piecewise constant and 
$\b=0, c=0$, the matrix generated by the FVM bilinear form $\AA_V(\cdot,\cdot)$ 
coincides with 
the matrix generated by the FEM bilinear form. Thus the FEM and FVM only 
differ in the right hand sides. 
See also~\cite[Sections 3.1 and 3.2]{Hackbusch:1989}.
\end{remark}

\subsection{Upwind stabilization} \label{section:upwind}
Here we will introduce the stabilization of FVM through an upwind scheme which
is mandatory to get a stable solution for convection dominated problems.
To define an upwind stabilization for FVM~\cite[Section 3.1]{Roos:2008-book} 
we simply replace
the terms with $\b u_h$  on the interior edges of the dual mesh by a 
convex combination of the nodal values depending on the direction of the 
convectional flux. 
On the intersection $\tau_{ij}=V_i\cap V_j\neq\emptyset$ of two neighboring cells 
we replace $u_h$  by
\begin{align}\label{eq:upwinddef}
u_{h,ij} := \lambda_{ij}u_h(a_i)+(1-\lambda_{ij})u_h(a_j).
\end{align}
The parameter $\lambda_{ij}$ is computed in the following way:
first we compute the average of the convection over the $\tau_{ij}$, i.e.,
\begin{align*}
\beta_{ij}:= \frac{1}{|\tau_{ij}|}\int_{\tau_{ij}} \b\cdot \normal_i \intd{s},
\end{align*}
where $\normal_i$ is the unit outer normal with respect to $V_i$, and the average of the diffusion
\begin{align*}
\A_{ij}:= \frac{1}{|\tau_{ij}|}\int_{\tau_{ij}} \A \intd{s}.
\end{align*}
Then $\lambda_{ij}$ is defined by
\begin{align*}
\lambda_{ij} := \Phi(\beta_{ij}|\tau_{ij}|/\norm{\A_{ij}}{\infty}),
\end{align*}
with a weight function $\Phi\colon\RR\to [0,1]$ determined by the used upwind scheme. 
The argument of this weight function is the local P{\'e}clet number, which describes the ratio
of the convection to the diffusion locally.
The easiest scheme is the full upwind scheme with $\Phi(t):= (\sgn(t)+1)/2$,
which leads to $u_{h,ij}=u_h(a_i)$ for $\beta_{ij} \geq 0$
and $u_{h,ij}=u_h(a_j)$ otherwise. Since the full uwpinding scheme is very diffusive,
another option is the steerable upwinding defined by
\begin{align*}
\Phi(t):= \begin{cases} \min(2|t|^{-1},1)/2,\quad &\text{for}~t<0, \\ 
1-\min(2|t|^{-1},1)/2,\quad &\text{for}~t\geq 0.  \end{cases}
\end{align*}
Replacing the respective term in the original 
finite volume bilinear form~\cref{eq:fvmbilinearform} by~\cref{eq:upwinddef}
leads to the upwind bilinear form 
(with $\mathcal{N}_i$ being the set of neighboring nodes of $a_i\in\mathcal{N}$):
\begin{align}
 \label{eq:fvmupwind}
 \begin{split}
\AA_V^{up}(u_h,v_h) := \sum_{a_i\in\mathcal{N}} v_h(a_i)  &\left( \int_{\partial V_i \setminus \Gamma} -\A\nabla u_h \cdot \normal \intd{s} + \int_{V_i} \c u_h \intd{x} \right. \\
& \left. + \sum_{j\in \mathcal{N}_i}\int_{\tau_{ij}}\b \cdot \normal u_{h,ij} \intd{s} + \int_{\partial V_i \cap \Gamma^\text{out}}\b \cdot \normal u_h \intd{s} \right).
\end{split}
\end{align}
%
\section{Semi-discretization}\label{sec:semi-disc}
In this section we allow \emph{model input data}
$q\in L^2(\Omega)$, $f\in L^2_{T,\Omega}$, $g_1\in B'_T$, and $g_2\in L^2_{T,\Gamma}$.
Similar to the semi-discretization with a
FEM-BEM coupling~\cite{Egger:2017}, we can 
also define a FVM-BEM coupling. 
More precisely, we replace the bilinear form 
in the first equation~\cref{eq:vp1} by the finite volume bilinear form and change the test 
space as seen in~\cref{eq:fvmformulation}.
Based on the continuous case we define the functional spaces
$H_T^h=L^2(0,T;\S^1(\T))$, $B_T^h=L^2(0,T;\P^0(\Er))$,
and the energy space $Q_T^h=\set{v_h\in H^1(0,T;\S^1(\T))}{v_h(0)=P_h q}$,
where $P_h:L^2(\Omega)\to \S^1(\T)$ denotes the $L^2$-orthogonal projection
defined by
\begin{align}
 \label{eq:L2projection}
 \product{P_h v}{w_h}{\Omega}=\product{v}{w_h}{\Omega}\quad\text{for all }w_h\in \S^1(\T).
\end{align} 
This results in the following semi-discrete problem.
\begin{problem}
Find $u_h(t)\in Q_T^h$ and $\phi_h(t) \in B_T^h$ such that
\begin{align}
\label{eq1:fvmbem}
  \product{\partial_t u_h(t)}{\IIh^* v_h}{\Omega} + \AA_V(u_h(t),v_h)-\dual{\phi_h(t)}{\IIh^* v_h}{\Gamma}
   &= \product{f(t)}{\IIh^*v}{\Omega}+\dual{g_2(t)}{\IIh^*v}{\Gamma},\\
   \label{eq2:fvmbem} 
   \dual{(1/2-\K)u(t)}{\psi}{\Gamma}+\dual{\V\phi(t)}{\psi}{\Gamma}
   &=\dual{(1/2 - \K) g_1(t)}{\psi}{\Gamma},
\end{align}
for all $v_h\in \S^1(\T),\psi_h\in \P^0(\Er)$ and a.e. $t\in [0,T]$.
Obviously, the bilinear form $\AA_V$ can be replaced by
the upwind bilinear form $\AA_V^{up}$.
\end{problem} 

With $\HH_T^h=Q_T^h\times B_T^h$ 
and $\HH^h=\S^1(\T)\times \P^0(\Er)$ we define the more compact bilinear form
$\BB_V\colon\HH^h_T\times\HH^h\to \RR$ by
\begin{align}
\begin{split}\label{eq:Bfvm}
	\BB_V((u_h(t),\phi_h(t));(v_h,\psi_h)):=   &\AA_V(u_h(t),v_h)-\dual{\phi_h(t)}{\IIh^* v_h}{\Gamma} \\
      &+\dual{(1/2-\K)u_h(t)}{\psi_h}{\Gamma}+\dual{\V\phi_h(t)}{\psi_h}{\Gamma},
  \end{split}
\end{align}
and the linear functional $F_V:\HH^h\to\RR$ by
\begin{align}
	\label{eq:Ffvm}
	F_V((v_h,\psi_h);t):=  \product{f(t)}{\IIh^*v_h}{\Omega}
 +\dual{g_2(t)}{\IIh^*v_h}{\Gamma}+\dual{(1/2-\K)g_1(t)}{\psi_h}{\Gamma}.  
\end{align}
Hence the system~\cref{eq1:fvmbem}--\cref{eq2:fvmbem} is equivalent to:
\begin{problem}
Find $\u_h(t)=(u_h(t),\phi_h(t))\in\HH_T^h$ such that
\begin{align}
	\label{eq:FVMBEMadded}
	\product{\partial_t u_h(t)}{\IIh^* v_h}{\Omega}+\BB_V(\u_h(t);\v_h)=F_V(\v_h;t)
\end{align}
for all $\v_h=(v_h,\psi_h)\in \HH^h$ and a.e. $t\in [0,T]$,
where we can replace $\AA_V$ by $\AA_V^{up}$ in $\BB_V$.
\end{problem} 

For the analysis of the system~\cref{eq:FVMBEMadded} we employ some results 
from the stationary FVM-BEM coupling~\cite{Erath:2017-1}.  
The main idea is to measure the discrete difference between
the right-hand sides and
the bilinear forms~\cref{eq:B} and~\cref{eq:Bfvm}:
\begin{lemma}[{\cite[Lemma 5]{Erath:2017-1}}]\label{lem:rhsdifference}
For $\w_h=(w_h,\varphi_h)\in \HH^h$ and an arbitrary but fixed $t$ there holds 
\begin{align*}
|F(\w_h;t)-F_V(\w_h;t))| 
\leq  C &\Big( \sum_{K\in\T} h_K \norm{f(t)}{L^2(K))}\norm{\nabla w_h}{L^2(K))}\\
&+ \sum_{E\in\Er} h_E^{1/2} 
\norm{g_2(t)-\overline{g}_2(t)}{L^2(E)}\norm{\nabla w_h}{L^2(K_E)}  \Big),
\end{align*}
with a constant $C>0$ independent of $h$.
Here, $\overline{g}_2(t)$ is the $\Er$-piecewise integral mean of $g_2(t)\in L^2(\Gamma)$
and $K_E\in\T$ the element associated with $E$.
\end{lemma}
\begin{lemma}[{\cite[Lemma 7]{Erath:2017-1}}]
\label{lem:bildifference}
For $\v_h=(v_h,\psi_h)\in \HH^h$ 
and $\w_h=(w_h,\varphi_h)\in \HH^h$ there holds
\begin{align*}
|\BB(\v_h;\w_h)-\BB_V(\v_h;\w_h)| \leq C \sum_{K\in\T} \left(h_K \norm{v_h}{H^1(K))}\norm{w_h}{H^1(K))}\right),
\end{align*}
with a constant $C>0$ independent of $h$. 
The result still holds if we replace $\AA_V$ by $\AA_V^{up}$ in the corresponding bilinear forms.
\end{lemma}

\begin{remark}
 The restriction $\b\cdot\normal \in \P^0(\Er^{in})$
 in \cite[Lemma 7]{Erath:2017-1}, where
 $\Er^{in}$ denotes the set of all edges on the inflow boundary $\Gamma^{in}$, 
 results from the estimate~\cite[Lemma 6]{Erath:2017-1}. However, this is not necessary.
 In fact, we can estimate the last term of~\cite[eq. (38)]{Erath:2017-1}
 in the following way: let $v_h,w_h\in\S^1(\T)$ and let
 $\overline{v}_h\in\P^0(\Er)$  be the best $L^2(\Gamma)$ approximation
  of $v_h$. We see with~\cref{eq:piecewiseintF}
 \begin{align*}
  -\sum_{E\in\Er^{in}}\product{\b\cdot\normal v_h}{w_h-\IIh^* w_h}{E}
  &\leq C\sum_{E\in\Er^{in}}\norm{\b}{L^{\infty}(E)}
  |\product{v_h-\overline{v}_h}{w_h-\IIh^* w_h}{E}|\\
  &\leq C\sum_{E\in\Er^{in}}\norm{\b}{L^{\infty}(E)} h_{K_E}
  \norm{v_h}{H^1(K_E)}\norm{w_h}{H^1(K_E)},
 \end{align*}
 where $K_E\in\T$ is the element associated with $E$.
 For the last estimate we used the Cauchy-Schwarz inequality,
 $\norm{v_h-\overline{v}_h}{L^2(E)}\leq C h_{K_E}^{1/2}\norm{\nabla v_h}{L^2(K_E)}$,
 and~\cref{eq:piecewisepropF}.
 The same applies for the stabilized FVM-BEM coupling versions with $\AA_V^{up}$
 and the three-field FVM-BEM coupling,
 where we neither need this restriction
 in~\cite[Lemma 5.2 and Theorem 5.3]{Erath:2012-1}.
\end{remark}
With~\cref{lem:bildifference} and the ellipticity of $\BB(\cdot;\cdot)$ we show:
\begin{lemma}[{\cite[Theorem 2]{Erath:2017-1}}]
\label{lem:fvmbem_ellipticity}
For $h$ small enough, let  $\lambda_{\min}(\A) - \frac{1}{4} C_{\K} > 0$,
where $C_\K\in [1/2,1)$ is the contraction constant of the double layer operator $\K$.
Then there holds
for all $\v_h=(v_h,\psi_h)\in \HH^h$
\begin{align}
 \label{eq:ellipticB}
 \BB_V(\v_h;\v_h)  &\geq C_{\textrm Vstab} \norm{\v_h}{\HH}^2=
  C_{\textrm Vstab}\Big(\norm{v_h}{H^1(\Omega)}^2 + \norm{\psi_h}{H^{-1/2}(\Gamma)}^2\Big).
\end{align}	
The constant $C_{\textrm Vstab}>0$ depends on the model 
data $\A$, $\b$, $\c$ and on $C_{\mathcal{K}}$.
The ellipticity still holds if we 
replace $\AA_V$ by $\AA_V^{up}$ in~\cref{eq:Bfvm}. 
Furthermore, the bilinear form is continuous.
\end{lemma}

\begin{remark}
\label{rem:ode}
The semi-discrete systems~\cref{eq1:fvmbem}--\cref{eq2:fvmbem} and~\cref{eq:FVMBEMadded}
lead to a system of ordinary differential equations
\begin{align*}
 M U_h'(t)+B\begin{pmatrix}
              U_h(t) \\ \Phi_h(t)
             \end{pmatrix} =F(t).
\end{align*}
Here, $U_h(t)\in\RR^{n_1}$, $\Phi_h(t)\in\RR^{n_2}$, and $F(t)\in\RR^{n_1+n_2}$ 
for some $n_1,n_2\in\NN$ and a fixed but
arbitrary $t$.
The matrix $B$ is positive definite which follows directly 
from~\cref{lem:fvmbem_ellipticity}.
The mass matrix $M$, resulting from $\product{\partial_t u_h}{\IIh^* v_h}{\Omega}$, 
is as well positive definite; see, e.g.,~\cite[Section 3.]{Chatzipantelidis:2004}.
Therefore, the ODE-system and thus also the semi-discrete system 
are uniquely solvable by the theorem of Picard-Lindel\"of. 
\end{remark}
Beside the unique solvability we also establish an energy estimate 
for the semi-discretization, which is similar to the result for the continuous problem.
\begin{lemma}[Well-posedness of the semi-discrete FVM-BEM]
For $h$ small enough, let  $\lambda_{\min}(\A) - \frac{1}{4} C_{\K} > 0$, $C_\K\in [1/2,1)$.
The 
solution $(u_h,\phi_h)\in\HH_T^h$ of~\cref{eq:FVMBEMadded} fulfills
\begin{align*}
 \norm{u_h}{H_T}&+ \norm{\phi_h}{B_T} 
 \leq  \norm{f}{L_{T,\Omega}^2} +\norm{q}{L^2(\Omega)} + \norm{g_2}{L_{T,\Gamma}^2} 
   + \norm{g_1}{B'_T}.
\end{align*}
\end{lemma}
\begin{proof}
 In~\cref{eq:FVMBEMadded} we choose $\v_h=(u_h(t),\phi_h(t))\in \HH^h$ for the test function
 for a fixed but arbitrary $t$.
 With 
 $\norm{u_h(t)}{L^2(\Omega)}^2\leq C\norm{u_h(t)}{\chi}^2=\product{u_h(t)}{\IIh^* u_h(t)}{\Omega}$
 from~\cref{lem:iih_norm}, the ellipticity~\cref{eq:ellipticB}
 of $\BB_V$, and the stability 
 $\norm{\IIh^* u_h(t)}{L^2(\Omega)}  \leq C \norm{u_h(t)}{L^2(\Omega)}$
 the result follows from standard calculations, 
 see, e.g.,~\cite[Section 7.1.2, Theorem 2]{Evans:2010-book}.
\end{proof}
The main result of this section is the following convergence of the semi-discrete scheme.
\begin{theorem}[Convergence of the semi-discrete FVM-BEM]\label{thm:semi-disc-conv}
There exists $h_{\max}>0$ such that for $\T$ sufficiently fine, i.e., $h<h_{\max}$, 
the following statement holds:
Let $\lambda_{\min}(\A) - \frac{1}{4} C_{\K} > 0$, $C_\K\in [1/2,1)$. 
The discrete solution $\u_h=(u_h,\phi_h)\in \HH_T^h=Q_T^h\times B_T^h$
of~\cref{eq:FVMBEMadded}
converges to the weak solution 
$\u=(u,\phi)\in \HH_T=Q_T\times B_T$ of~\cref{eq:weakAddedForm} , i.e., 
there holds
\begin{align*}
  \norm{\u-\u_h}{\HH_T}&\leq C \Big[h\norm{f}{L_{T,\Omega}^2}
  +h^{1/2}\norm{g_2-\overline{g}_2}{L_{T,\Gamma}^2}
  + h\norm{v_h}{H_T} \\
    &\qquad+ h\norm{\partial_t v_h}{L_{T,\Omega}^2}
     +\norm{\partial_t u - \partial_t v_h}{H_T'} 
     +\norm{\u-\v_h}{\HH_T} \Big]. 
\end{align*}
for all $\v_h=(v_h,\psi_h) \in \HH_T^h$ and with $\overline{g}_2$ 
being the $\Er$-piecewise integral
mean of the normal derivative jump $g_2\in L_{T,\Gamma}^2$.
The constant $C>0$ depends on the model parameters and the shape regularity constant but not on $h$.
The result still holds if we replace $\AA_V$ by $\AA_V^{up}$ in the corresponding bilinear forms.
\end{theorem}
\begin{proof}
 Let $\v_h=(v_h,\psi_h) \in \HH_T^h$ be arbitrary.
 First we split the error into an approximation error and a discrete error component;
 \begin{align}
  \label{eq:splitsemierror}
  \begin{split}
  \norm{\u-\u_h}{\HH} &\leq  \norm{\u-\v_h}{\HH} + \norm{\u_h-\v_h}{\HH} .
  \end{split} 
 \end{align}
 Hence, we only have to estimate the norms of the discrete error
 $\w_h=(w_h,\varphi_h):=\u_h-\v_h\in \HH_T^h$.
 Since $\IIh^*$ is self-adjoint~\cref{eq:selfadjoint} and defines a norm~\cref{eq:piecewiseeqNorm}
 we see that
 $\frac{1}{2}\partial_t \norm{w_h}{\chi}^2=\product{\partial_t w_h}{\IIh^* w_h}{\Omega}$. 
 The ellipticity~\cref{eq:ellipticB} of the finite volume 
 bilinear form $\BB_V(\cdot;\cdot)$ leads to
 \begin{align*}
   \frac{1}{2}\partial_t \norm{w_h}{\chi}^2 
   + \norm{\w_h}{\HH}^2 
   &\lesssim \product{\partial_t w_h}{\IIh^* w_h}{\Omega} 
   + \BB_V(\w_h;\w_h). 
 \end{align*}
 Using the discrete FVM-BEM scheme~\cref{eq:FVMBEMadded} and adding
 the weak form~\cref{eq:weakAddedForm}
 we see
\begin{align*}
  \frac{1}{2}\partial_t \norm{w_h}{\chi}^2 + \norm{\w_h}{\HH}^2
  &\leq  \dual{\partial_t u}{w_h}{\Omega}
  - \product{\partial_t v_h}{\IIh^*w_h}{\Omega}
  + F_V(\w_h;t) - F(\w_h;t) \\
  &\quad+ \BB(\v_h;\w_h) 
  - \BB_V(\v_h;\w_h) + \BB(\u - \v_h;\w_h).
\end{align*}
To estimate the terms with the time derivatives we apply~\cref{eq:piecewisepropT}:
\begin{align}
 \label{eq:timederivativeerror}
 \begin{split}
  \dual{\partial_t u}{w_h}{\Omega} 
  &- \product{\partial_t v_h}{\IIh^*w_h}{\Omega} \\
  &= \product{\partial_t v_h}{w_h - \IIh^*w_h}{\Omega} 
  +\dual{\partial_t u - \partial_t v_h}{w_h}{\Omega} \\
  &\lesssim h\norm{\partial_t v_h}{L^2(\Omega)}\norm{\nabla w_h}{L^2(\Omega)} + \norm{\partial_t u - \partial_t v_h}{H^1(\Omega)'}\norm{w_h}{H^1(\Omega)}.
  \end{split}
\end{align}
We estimate the other terms by~\cref{lem:rhsdifference},~\cref{lem:bildifference},
and the continuity of the bilinear form $\BB$. Thus we get
\begin{align*}
 \frac{1}{2}\partial_t \norm{w_h}{\chi}^2 + \norm{\w_h}{\HH}^2 
 &\lesssim \Big(\norm{\partial_t u - \partial_t v_h}{H^1(\Omega)'}
 +h\norm{\partial_t v_h}{L^2(\Omega)}
 +h\norm{f}{L^2(\Omega)}\\ 
 &\qquad+ h^{1/2}\norm{g_2-\overline{g}_2}{L^2(\Gamma)} 
 + h\norm{v_h}{H^1(\Omega)} 
 +\norm{\u-\v_h}{\HH}\Big)\norm{\w_h}{\HH}. 
\end{align*}
Young's inequality with $\varepsilon>0$,
integration over $t$ from $0$ to $T$, and the fact that
\begin{align*}
\int_0^T \frac{1}{2}\partial_t \norm{w_h}{\chi}^2 \intd{t} 
= \underbrace{\frac{1}{2}\norm{w_h(T)}{\chi}^2}_{\geq 0} 
-  \underbrace{\frac{1}{2}\norm{w_h(0)}{\chi}^2}_{=0} \geq 0. 
\end{align*}
lead to
\begin{align*}
  \norm{\w_h}{\HH_T}^2 
  &\leq C\frac{1}{\varepsilon} 
  \Big[\norm{\partial_t u - \partial_t v_h}{H_T'}^2 
  + \norm{\u-\v_h}{\HH_T}^2 + h\norm{g_2-\overline{g}_2}{L_{T,\Gamma}^2}^2 \\
  &\quad +  h^2\Big(\norm{f}{L_{T,\Omega}^2}^2 
  + \norm{v_h}{H_T}^2 
  + \norm{\partial_t v_h}{L_{T,\Omega}^2}^2\Big)\Big] 
  + C\varepsilon\norm{\w_h}{\HH_T}^2.
\end{align*}
We consequently choose $\varepsilon>0$ such that $C\varepsilon\leq 1/2$
and conclude the assertion with $\w_h=\u_h-\v_h$
and the error splitting~\cref{eq:splitsemierror}. 
For the stabilized FVM-BEM coupling version with $\AA_V^{up}$ 
the proof is the same.
\end{proof}
Before we state an a~priori estimate we recall the following 
approximation results~\cite[Lemma 20]{Egger:2017}.
 We denote by 
 $P_h: L^2(\Omega) \to \S^1(\T)$ and $\Pi_h : H^{-1/2}(\Gamma) \to \P^0(\Er)$ the 
 $L^2(\Omega)$- and the $H^{-1/2}(\Gamma)$-orthogonal projection, respectively.  
 Besides the $L^2$-stability we also require the $H^1$-stability of $P_h$ for
 the next corollary to hold.

\begin{remark}
 \label{rem:H1stable}
 We say $P_h$ is $H^1$-stable if there exists a constant $C_P>0$ such that
 $\norm{P_h v}{H^1(\Omega)}\leq C_P\norm{v}{H^1(\Omega)}$ for all $v\in H^1(\Omega)$.
 If $\T$ is quasi-uniform $H^1$-stability for our chosen function space 
 follows via an inverse inequality. For more general meshes and details we
 refer to~\cite[Remark 21]{Egger:2017}.
\end{remark}
\begin{lemma}
 \label{lem:approx}
 Let $P_h$ be $H^1$-stable, e.g., $\T$ is quasi-uniform.
 The operator $P_h$ can be extended to a bounded linear operator on $H^1(\Omega)'$.
Hence, for all $0 \le s \le 1$ and $0 \le s_e \le 3/2$ we have
 \begin{align*}
 \norm{v - P_h v}{H^{1}(\Omega)} &\le C h^s \norm{v}{H^{1+s}(\Omega)},\qquad v\in H^{1+s}(\Omega), \\
 \norm{v - P_h v}{H^{1}(\Omega)'} &\le C h^s \norm{v}{H^{1-s}(\Omega)'},\qquad v\in H^{1-s}(\Omega)',  \\
 \norm{\psi - \Pi_h \psi}{H^{-1/2}(\Gamma)} &\le C h^{s_e} 
 \norm{\psi}{H^{-1/2+s_e}(\Gamma)},\qquad \psi\in H^{-1/2+s_e}(\Gamma).
 \end{align*}
 The constant $C>0$ is independent of the particular choice of the triangulation.
\end{lemma}
\begin{corollary}[Convergence rates of the semi-discrete FVM-BEM]
 \label{cor:convordersemi}
Let $P_h$ be $H^1$-stable, e.g., $\T$ is quasi-uniform. 
With the assumptions of~\cref{thm:semi-disc-conv} we obtain
 \begin{align*}
    \norm{\u-\u_h}{\HH_T}\leq C\Big[& h^s 
    \Big( \norm{u}{L^2(0,T;H^{1+s}(\Omega))}+\norm{\partial_t u}{L^2(0,T;H^{1-s}(\Omega)')} 
    + \norm{\phi}{L^2(0,T;H^{s-1/2}(\Gamma))}\\
    &\quad+\norm{g_2}{L^2(0,T;L^2(\Gamma)\cap H^{s-1/2}(\Gamma))}\Big )\\
     + & h\Big(\norm{f}{L_{T,\Omega}^2} 
    +\norm{u}{L_{T,\Omega}^2}
    +\norm{\partial_t u}{L_{T,\Omega}^2}\Big)\Big]=\O(h^s).
 \end{align*}
 for all $0\leq s\leq 1$, $u(t)\in H^{1+s}(\Omega)$, $\partial_t u(t)\in L^2(\Omega)$,
 $\phi(t)\in H^{-1/2+s}(\Gamma)$,
 $g_2(t)\in L^2(\Gamma)\cap H^{-1/2+s}(\Gamma)$, and for a.e. $t\in [0,T]$.
 
\end{corollary}
\begin{proof}
The result follows directly from~\cref{thm:semi-disc-conv} and~\cref{lem:approx}
with $v_h=P_h u$ and $\psi_h=\Pi_h\phi$.
\end{proof}

\begin{remark}
 \label{rem:BEMregular}
 In~\cref{cor:convordersemi} it is enough to demand $\phi(t)\in H^{-1/2+s}(\Er)$ and
 $g_2\in H^{-1/2+s}(\Er)$ if $s>1/2$. More precisely, 
 for $\psi\in H^{s_e}(\Er)$, $0\leq s_e\leq 1$,
 with $H^{s_e}(\Er):=\set{\psi\in L^2(\Gamma)}{\psi|_E\in H^{s_e}(E)\text{ for all }E\in\Er}$,
 there holds 
 $\norm{\psi-\Pi_h \psi}{H^{-1/2}(\Gamma)}\leq C h^{1/2+se}\norm{\psi}{L^2(\Er)}$ 
 with $\norm{\psi}{H^{s_e}(\Er)}^2=\sum_{i=1}^{\#\Er}\norm{v|_E}{H^s(E)}^2$.
\end{remark}

\section{Full-discretization}
\label{sec:fully-disc}
In~\cref{sec:semi-disc} we introduced a FVM-BEM coupling for a discretization
of the model problem~\cref{eq:model1}--\cref{eq:model6} in space.
This semi-discretization leads to 
a stiff system of ordinary differential equations, see~\cref{rem:ode}. 
The advantage of this \emph{method of lines} approach is that we can choose between
several time discretization schemes.
In this section we analyze the subsequent time discretization of this system
by an implicit scheme. 
We introduce a variant of the backward Euler scheme
which allows us to present an analysis under minimal regularity assumptions
but with a slightly more expensive right-hand side.
Furthermore, we define a fully discrete system with the aid of a classical
backward Euler scheme for time discretization where we demand the usual regularity for the
time component of the model data and solution.

Let us first divide the time interval $[0,T]$ into $N\in\NN$ time-steps, i.e,
$0=t^0<t^1<\ldots<t^N=T$. Then $\tau^n=t^n-t^{n-1}$ is the local time step
and $\tau=\max_{n=1,\ldots,N}\tau^n$.
For a smooth enough function $v$ we write $v^n:=v(t^n)$ for the function evaluation at $t^n$. 
Consequently, we abbreviate the discrete time derivative by
\begin{align*}
 \dtau v^{n} := \frac{1}{\tau^n}\big(v^{n}- v^{n-1} \big).
\end{align*}

\subsection{A variant of the backward Euler scheme}
\label{sec:variantEuler}
In this section a special time discretization allows us to analyze
a fully discrete system with minimal regularity assumptions on
the model solution $\u=(u,\phi)\in \HH_T=Q_T\times B_T$ of~\cref{eq:weakAddedForm}.
The \emph{model input data} are still 
$q\in L^2(\Omega)$, $f\in L^2_{T,\Omega}$, $g_1\in B'_T$, and $g_2\in L^2_{T,\Gamma}$.
Let us note that a similar method was used in~\cite[Section 4.1.]{Tantardini:2014-1} for the discretization
of a parabolic problem and in~\cite[Section 4.]{Egger:2017} for a parabolic-elliptic problem with a FEM-BEM
discretization in space.
In contrast, a classical approach for the time analysis
from the literature requires slightly higher regularity
in the time component, but is computationally cheaper, see also \cref{sec:classicalEuler}.
Hence, we search for functions $u_{h,\tau} \in Q_T^{h,\tau}$
and $\phi_{h,\tau} \in B_T^{h,\tau}$ with 
\begin{align*}
Q_T^{h,\tau} 
&:=\set{v  \in C(0,T;\S^1(\T))}{v(0) = P_h q, v|_{[t^{n-1},t^n]} \text{ is linear in } t} \qquad \text{and} \\
B_T^{h,\tau} 
&:= \set{\psi  \in L^2(0,T;\P^0(\Er))}{\psi|_{(t^{n-1},t^n]} \text{ is constant in } t}.
\end{align*}
The notation in product space reads 
$\u_{h,\tau}=(u_{h,\tau},\phi_{h,\tau})\in \HH_T^{h,\tau}:= Q_T^{h,\tau}\times B_T^{h,\tau}$.
For $u_{h,\tau}\in Q_T^{h,\tau}$ the operator $\dt$ has to be understood
piecewise with respect to the time mesh, in particular, there holds
\begin{align}
 \label{eq:discderivative}
  \dt u_{h,\tau}|_{(t^{n-1},t^n)}=\dtau u_{h,\tau}^n
 \qquad\text{with}\qquad\dtau u_{h,\tau}^n := \frac{1}{\tau^n} (u_{h,\tau}^n - u_{h,\tau}^{n-1}).
\end{align}
We further introduce weighted averages
\begin{align} \label{eq:hatv}
\widehat{v}^n = \frac{1}{\tau^n} \int_{t^{n-1}}^{t^n} v(t) \omega^n(t)\,dt
\qquad\text{with}\qquad\omega^n(t)=\frac{6t-2t^n-4t^{n-1}}{\tau^n}
\end{align}
and define our fully discrete system as follows.
\begin{problem}[VarBE-FVM-BEM]
\label{prob:variantEuler}
Find $u_{h,\tau} \in Q_T^{h,\tau}$ and $\phi_{h,\tau} \in B_T^{h,\tau}$ such that 
\begin{align}
 \label{eq:vp1htauweight}
\product{\widehat{\dt u}_{h,\tau}^n}{\IIh^* v_h}{\Omega}
+ \AA_V(\widehat{u}_{h,\tau}^n,v_h)
- \product{\widehat{\phi}_{h,\tau}^n}{\IIh^* v_h}{\Gamma} 
&= \product{\widehat{f}^n}{\IIh^* v_h}{\Omega}
+\product{\widehat{g}^n_2}{\IIh^* v_h}{\Gamma},\\
\label{eq:vp2htauweight}
\product{(1/2-\K) \widehat{u}_{h,\tau}^n}{\psi_h}{\Gamma}
+ \product{\V\widehat{\phi}_{h,\tau}^n}{\psi_h}{\Gamma}
&= \product{(1/2-\K)\widehat{g}^n_1}{\psi_h}{\Gamma}
\end{align}
for all $v_h \in \S^1(\T) \subset H^1(\Omega)$ and $\psi_h \in \P^0(\T) \subset H^{-1/2}(\Gamma)$
and for all $1 \le n \le N$.\\
In compact notation: Find $\u_{h,\tau}=(u_{h,\tau},\phi_{h,\tau})\in \HH_T^{h,\tau}$
such that
\begin{align}
 \label{eq:fullytimeadded}
\product{\widehat{\dt u}_{h,\tau}^n}{\IIh^* v_h}{\Omega}
+ \BB_V(\widehat{\u}_{h,\tau}^n;\v_h)=\widehat{F}_V(\v_h;t)
\end{align}
for all $\v_h=(v_h,\psi_h)\in\HH^h$. Here, $\widehat{F}_V$
is the $\omega$-weighted average~\cref{eq:hatv} of $F_V$
defined in~\cref{eq:Ffvm}.
In~\cref{eq:vp1htauweight} and in~\cref{eq:fullytimeadded} we can
replace $\AA_V$ by $\AA_V^{up}$.
\end{problem}
The next lemma emphasizes the interpretation of~\cref{prob:variantEuler}
as a variant of a classical backward Euler time discretization. 
\begin{lemma}[{\cite[Section 4]{Egger:2017}}]
 Choose  $\omega^n(t)=\frac{6t-2t^n-4t^{n-1}}{\tau^n}$ as the linear weight function,
 For all
 $n\in\NN$, $v_{h,\tau} \in Q_T^{h,\tau}$, and $\psi_{h,\tau} \in B_T^{h,\tau}$ there holds 
 \begin{align}
  \label{eq:widentities}
 \widehat{v}_{h,\tau}^n = v_{h,\tau}^n, \qquad 
 \widehat{\dt v}_{h,\tau}^n =  \dtau v_{h,\tau}^n
 = \frac{1}{\tau^n} (v_{h,\tau}^n - v_{h,\tau}^{n-1}), 
 \quad \text{ and} \quad
 \widehat{\psi}_{h,\tau}^n = \psi_{h,\tau}^n.
 \end{align}
 Since $v_{h,\tau}$ and $\psi_{h,\tau}$ are piecewise linear and constant, respectively,
 we easily see that
 \begin{align}
  \label{eq:winequality1}
  \norm{v_{h,\tau}}{H_T}^2\leq \frac{4}{3}\sum_{n=1}^N \tau^n \norm{v_{h,\tau}^{n}}{H^1(\Omega)}^2
  \quad\text{ and }\quad
   \norm{\psi_{h,\tau}}{B_T}^2\leq \sum_{n=1}^N \tau^n \norm{\psi_{h,\tau}^{n}}{H^{-1/2}(\Omega)}^2.
 \end{align}
Furthermore, for any $v \in L^2(0,T;X)$ with values in some Hilbert space $X$, the
Cauchy-Schwarz inequality and $\norm{\omega^n(t)}{L^2(t^{n-1},t^n)}=4\tau^n$ lead to
\begin{align}
 \label{eq:winequality2}
\sum_{n=1}^N \tau^n \norm{\widehat{v}^n}{X}^2 \le 4 \norm{v}{L^2(0,T;X)}^2. 
\end{align} 
\end{lemma}

\begin{remark}
 \label{rem:classicalEuler}
 With the identities~\cref{eq:widentities},
 the discrete system \cref{eq:vp1htauweight}--\cref{eq:vp2htauweight} 
 is equivalent to 
 \begin{align}
  \label{eq:vp1htau}
 \product{\dtau u_{h,\tau}^n}{\IIh^* v_h}{\Omega} 
 + \AA_V(\nabla u_{h,\tau}^n,\IIh^* v_h) 
 - \product{\phi_{h,\tau}^n}{\IIh^* v_h}{\Gamma} 
 &= \product{\widehat{f}^n}{\IIh^* v_h}{\Omega}+
    \product{\widehat{g}^n_2}{\IIh^* v_h}{\Gamma},\\
 \label{eq:vp2htau}
 \product{(1/2-\K) u_{h,\tau}^n} {\psi_h}{\Gamma} 
 + \product{\V\phi_{h,\tau}^n}{\psi_h}{\Gamma} 
 &= \product{(1/2-\K)\widehat{g}^n_1}{\psi_h}{\Gamma} 
 \end{align}
 for all $v_h \in \S^1{(\T)} \subset H^1(\Omega)$ and $\psi_h \in \P^0(\Er) \subset H^{-1/2}(\Gamma)$,
  and for all $1 \le n \le N$. The same holds if we replace $\AA_V$ by $\AA_V^{up}$.
  This system differs from a time discretization by 
  a classical backward Euler only in the right-hand side, cp. \cref{eq:vp1htauclassic}--\cref{eq:vp2htauclassic}
  in \cref{sec:classicalEuler}.
\end{remark}
As in~\cite{Egger:2017} we rewrite the variational form~\cref{eq:vp1}--\cref{eq:vp2} 
to see that the fully discrete
system~\cref{eq:vp1htauweight}--\cref{eq:vp2htauweight} is consistent.
 More precisely, by testing~\cref{eq:vp1}--\cref{eq:vp2} with $v=v_h$ and $\psi=\psi_h$, 
 multiplication with the weight function $\omega^n$, and 
 integration over the time interval $[t^{n-1},t^n]$, we see that  
 \begin{align*}
 \dual{\widehat{\dt u}^n}{v_h}{\Omega} + \AA(\widehat{ u}^n,v_h)
 - \dual{\widehat{\phi}^n}{v_h}{\Gamma} 
 &= \dual{\widehat{f}^n}{v_h}{\Omega}+\dual{\widehat{g}^n_2}{v_h}{\Gamma},\\
 \dual{(1/2-\K) \widehat{u}^n}{\psi_h}{\Gamma} + \dual{\V  \widehat{\phi}^n}{\psi_h}{\Gamma} 
 &= \dual{(1/2-\K)\widehat{g}^n_1}{\psi_h}{\Gamma}
 \end{align*}
 for all $v_{h} \in \S^1(\T)$, $\psi_h \in \P^0(\Er)$. 
 We write this system in the compact form with $\u=(u,\phi)\in\HH_T$
 \begin{align}
  \label{eq:variationhat}
  \dual{\widehat{\dt u}^n}{v_h}{\Omega} + \BB(\widehat{ \u}^n;\v_h)=\widehat{F}(\v_h;t)
 \end{align}
 for all $\v_h \in \HH^h$, where $\widehat{F}$
is the $\omega$-weighted averaged~\cref{eq:hatv} of $F$ defined in~\cref{eq:F}.
\begin{lemma}[Well-posedness of the fully discrete system VarBE-FVM-BEM]
For $h$ small enough, let  $\lambda_{\min}(\A) - \frac{1}{4} C_{\K} > 0$,
$C_{\K}\in [1/2,1)$.
The 
solution 
$\u_{h,\tau}=(u_{h,\tau},\phi_{h,\tau})\in\HH^{h,\tau}_T=Q^{h,\tau}_T\times B^{h,\tau}_T$ 
of~\cref{eq:fullytimeadded} fulfills
\begin{align*}
 \norm{\u_{h,\tau}}{\HH_T} 
 \leq  \norm{f}{L_{T,\Omega}^2} +\norm{q}{L^2(\Omega)} + \norm{g_2}{L_{T,\Gamma}^2} 
   + \norm{g_1}{B'_T}.
\end{align*}
\end{lemma}
\begin{proof}
 For a time $t^n$ we estimate
 \begin{align}
  \label{eq:dtdiscrete}
  \begin{split}
 \product{\dtau u^n_{h,\tau}}{\IIh^* u^n_{h,\tau}}{\Omega}
 &\geq \frac{1}{\tau^n}\product{u^n_{h,\tau} - u^{n-1}_{h,\tau}}{\IIh^* u^n_{h,\tau}}{\Omega} 
 - \frac{1}{2\tau^n}\norm{u^n_{h,\tau}-u^{n-1}_{h,\tau}}{\chi}^2 \\
 &=\frac{1}{2\tau^n}
 \product{u^n_{h,\tau}
 -u^{n-1}_{h,\tau}}{\IIh^*(u^n_{h,\tau}+u^{n-1}_{h,\tau})}{\Omega} \\
 &=\frac{1}{2\tau^n}(\norm{u^n_{h,\tau}}{\chi}^2 - \norm{u^{n-1}_{h,\tau}}{\chi}^2),
 \end{split}
 \end{align}
 where we used the fact that $\IIh^*$ is self-adjoint~\cref{eq:selfadjoint}.
 The assertion follows with standard arguments from the equivalent
 discrete system~\cref{eq:vp1htau}--\cref{eq:vp2htau} with $v_h=u_{h,\tau}^n$,~\cref{eq:dtdiscrete},
 the ellipticity of $\BB_V$, 
 $\norm{\IIh^* u^n_{h,\tau}}{L^2(\Omega)}\leq C\norm{u^n_{h,\tau}}{L^2(\Omega)}$,
 and~\cref{eq:winequality1}--\cref{eq:winequality2}.
\end{proof}

\begin{theorem}[Convergence of the fully discrete system VarBE-FVM-BEM]
 \label{thm:convfullydiscrete}
 There exists $h_{\max}>0$ such that for $\T$ sufficiently fine, i.e., $h<h_{\max}$, 
 the following statement holds:
 Let $\lambda_{\min}(\A) - \frac{1}{4} C_{\K} > 0$. 
 For the solution $\u=(u,\phi)\in\HH_T=Q_T\times B_T$ 
 of our model problem~\cref{eq:weakAddedForm} and the discrete solution
 $\u_{h,\tau}=(u_{h,\tau},\phi_{h,\tau})\in\HH_T^{h,\tau}=Q_T^{h,\tau}\times B_T^{h,\tau}$
 of our fully discrete system~\cref{eq:fullytimeadded}
 there holds
 \begin{align*}
  \norm{\u-\u_{h,\tau}}{\HH_T} &\lesssim
  \norm{\u-\v_{h,\tau}}{\HH_T} +\norm{\dt u - \dt v_{h,\tau}}{H_T'}
  +h\norm{\dt v_{h,\tau}}{L_{T,\Omega}^2}\\
  &\quad
  + h\norm{v_{h,\tau}}{H_T}+h\norm{f}{L_{T,\Omega}^2} 
  + h^{1/2}\norm{g_2-\overline{g}_2}{L_{T,\Gamma}^2} 
 \end{align*}
 for all $\v_{h,\tau}=(v_{h,\tau},\psi_h)\in\HH_T^{h,\tau}$, where
 $\overline{g}_2\in L^2(0,T;\P^0(\Er))$ is the $\Er$-piecewise integral mean
 of $g_2\in L_{T,\Gamma}^2$.
 This result also holds if we replace $\AA_V$ by the upwind version $\AA_V^{up}$.
\end{theorem}

\begin{proof}
 The proof uses results and techniques from the proof of~\cref{thm:semi-disc-conv}
 and~\cite[Lemma 14]{Egger:2017}.
 First we split the error into an approximation error and a discrete error component,
 i.e., for arbitrary 
 $\v_{h,\tau}=(v_{h,\tau},\psi_{h,\tau})\in\HH_T^{h,\tau}=Q_T^{h,\tau}\times B_T^{h,\tau}$
 \begin{align}
  \label{eq:errorsplitting}
  \norm{\u-\u_{h,\tau}}{\HH_T} 
  &\leq  \norm{\u-\v_{h,\tau}}{\HH_T} + \norm{\u_{h,\tau}-\v_{h,\tau}}{\HH_T}.
 \end{align}
 We only have to estimate the discrete error part. 
 With the notation
 $\w_{h,\tau}=(w_{h,\tau},\varphi_{h,\tau}):=\u_{h,\tau}-\v_{h,\tau}\in\HH_T^{h,\tau}$
 we estimate for a time $t^n$ as in~\cref{eq:dtdiscrete}
 \begin{align*}
 \product{\dtau w^n_{h,\tau}}{\IIh^* w^n_{h,\tau}}{\Omega}
 &\geq\frac{1}{2\tau^n}(\norm{w^n_{h,\tau}}{\chi}^2 - \norm{w^{n-1}_{h,\tau}}{\chi}^2).
 \end{align*}
 With this estimate and the ellipticity~\cref{eq:ellipticB}
 of $\BB_V$ we get with similar steps as for the proof of \cref{thm:semi-disc-conv}
 \begin{align*}
  &\frac{1}{2\tau^n}(\norm{w^n_{h,\tau}}{\chi}^2 - \norm{w^{n-1}_{h,\tau}}{\chi}^2)
  +\norm{\w^n_{h,\tau}}{\HH}^2\\
  &\quad\lesssim
  \product{\dtau w^n_{h,\tau}}{\IIh^*w^n_{h,\tau}}{\Omega}
  +\BB_V(\w^n_{h,\tau};\w^n_{h,\tau})\\
  &\quad\lesssim  \dual{\widehat{\dt u}^n}{w^n_{h,\tau}}{\Omega}
  - \product{\dtau v^n_{h,\tau}}{\IIh^*w^n_{h,\tau}}{\Omega}
  + \widehat F_V(\w^n_{h,\tau};t) - \widehat F(\w^n_{h,\tau};t) \\
  &\qquad+ \BB(\v^n_{h,\tau};\w^n_{h,\tau}) 
  - \BB_V(\v^n_{h,\tau};\w^n_{h,\tau}) 
  + \BB(\widehat\u - \v^n_{h,\tau};\w^n_{h,\tau})\\
  &\quad\lesssim \Big(\norm{\widehat{\dt u}^n - \dtau v^n_{h,\tau}}{H^1(\Omega)'}
  +h\norm{\dtau v^n_{h,\tau}}{L^2(\Omega)}
  +h\norm{\widehat f}{L^2(\Omega)} 
  + h^{1/2}\norm{\widehat{g}_2-\widehat{\overline{g}}_2}{L^2(\Gamma)}\\ 
  &\qquad \quad + h\norm{v^n_{h,\tau}}{H^1(\Omega)} 
  +\norm{\widehat \u-\v^n_{h,\tau}}{\HH} \Big)
  \norm{\w^n_{h,\tau}}{\HH},
 \end{align*}
 where we used the discrete system~\cref{eq:fullytimeadded} with $\widehat{\u}^n_{h,\tau}=\u_{h,\tau}$ 
 and the $\omega$-weighted variational form~\cref{eq:variationhat}.
In the last step we 
used~\cref{eq:timederivativeerror},~\cref{lem:rhsdifference},~\cref{lem:bildifference},
 and the continuity of the bilinear form $\BB$.
Young's inequality with $\epsilon>0$, multiplying the whole inequality with $\tau^n$
and summing over $n$ lead to
\begin{align*}
  &\frac{1}{2}\norm{w^N_{h,\tau}}{\chi}^2
  +\sum_{n=1}^N\tau^n\norm{\w^n_{h,\tau}}{\HH}^2\\
  &\quad\lesssim \sum_{n=1}^N\tau^n
  \Big(\norm{\widehat{\dt u}^n - \dtau v^n_{h,\tau}}{H^1(\Omega)'}^2
  +h^2\norm{\dtau v^n_{h,\tau}}{L^2(\Omega)}^2
  +h^2\norm{\widehat f}{L^2(\Omega)}^2 \\
  &\qquad\qquad\qquad+ h\norm{\widehat{g}_2-\widehat{\overline{g}}_2}{L^2(\Gamma)}^2 + h^2\norm{v^n_{h,\tau}}{H^1(\Omega)}^2 
  +\norm{\widehat \u-\v^n_{h,\tau}}{\HH}^2 \Big).
\end{align*} 
 Finally, we estimate with $\dtau v^n_{h,\tau}=\widehat{\partial_t v}^n_{h,\tau}$
 and $v^n_{h,\tau}=\widehat{v}^n_{h,\tau}$ from~\cref{eq:widentities}, 
 and the inequalities~\cref{eq:winequality1}--\cref{eq:winequality2}
 \begin{align*}
  \norm{\w_{h,\tau}}{\HH_T}^2 & \lesssim
  \norm{\dt u - \dt v_{h,\tau}}{H_T'(\Omega)'}^2
  +h^2\norm{\dt v_{h,\tau}}{L_{T,\Omega}^2}^2
  +h^2\norm{f}{L_{T,\Omega}^2}^2 \\
  &\quad+ h\norm{g_2-\overline{g}_2}{L_{T,\Gamma}^2}^2 
  + h^2\norm{v_{h,\tau}}{H_T(\Omega)}^2 
  +\norm{\u-\v_{h,\tau}}{\HH_T}^2 
 \end{align*}
 With $\w_{h,\tau}=\u_{h,\tau}-\v_{h,\tau}$ and~\cref{eq:errorsplitting}
 we prove the assertion.
\end{proof}

For the full discretization we also require the $L^2$-projection in time, i.e., 
we define the operators 
\begin{align*}
P^\tau &: Q_T \to 
Q_T^\tau:=\set{v\in Q_T}{v|_{[t^{n-1},t^n]}\text{ is linear in }t},\\
\Pi^\tau &: B_T \to 
B_T^\tau:=\set{\psi\in B_T}{\psi|_{(t^{n-1},t^n]}\text{ is constant in }t}.
\end{align*}
For sufficiently smooth functions 
these satisfy
\begin{align*}
\norm{v - P^\tau v}{Q_T} &\le C \tau^r \big(\norm{\dt v}{H^r(0,T;H^1(\Omega)')}
+ \norm{v}{H^r(0,T;H^1(\Omega))}\big),\qquad 0\leq r\leq 1, \\
\norm{\psi - \Pi^\tau \psi}{B_T} &\le C \tau^r \norm{\psi}{H^r(0,T;H^{-1/2}(\Gamma))},\qquad 0\leq r\leq 1.
\end{align*}

With the estimates for the projection $P_h$ and $\Pi_h$ in~\cref{lem:approx} and the estimates
for $P^\tau$ and $\Pi^\tau$ the following corollary is valid if $P_h$ is $H^1$-stable, see~\cref{rem:H1stable}.
\begin{corollary}[A~priori estimate for the fully discrete system VarBE-FVM-BEM]
 \label{cor:convratefullydisc}
 Let $P_h$ be $H^1$-stable, e.g., $\T$ is quasi-uniform.
 With the assumptions of~\cref{thm:convfullydiscrete} there holds
 \begin{align*}
  \norm{\u-\u_{h,\tau}}{\HH_T}  &\lesssim
 h^s \Big( \norm{u}{L^2(0,T;H^{1+s}(\Omega))} + \norm{\dt u}{L^2(0,T;H^{1-s}(\Omega)')} 
 + \norm{\phi}{L^2(0,T;H^{s-1/2}(\Gamma))}  \\
 &\quad+\norm{g_2}{L^2(0,T;L^2(\Gamma)\cap H^{-1/2+s}(\Gamma))}\Big)
 +h\norm{\dt u}{L_{T,\Omega}^2}+ h\norm{u}{H_T}+h\norm{f}{L_{T,\Omega}^2}\\ 
 & \quad +\tau^r \Big( \norm{\dt u}{H^r(0,T;H^1(\Omega)')} + \norm{u}{H^r(0,T;H^1(\Omega))} 
 + \norm{\phi}{H^r(0,T;H^{-1/2}(\Gamma))} \Big)\\
 &= \O(h^s+\tau^r)
 \end{align*}
 for all $0 \le s \le 1$ and $0 \le r \le 1$ with
 $u\in H^r(0,T;H^{1+s}(\Omega))$, $\dt u\in H^r(0,T;L^2(\Omega))$, and $\phi\in H^r(0,T;H^{-1/2+s}(\Gamma))$, 
 and  $g_2(t)\in L^2(0,T;L^2(\Gamma)\cap H^{-1/2+s}(\Gamma))$.
The (hidden) constant depends only on the domain $\Omega$ and the time horizon $T$.
\end{corollary}
\begin{proof}
 The proof follows the lines of the proof of~\cite[Theorem 24]{Egger:2017}
 and uses \cref{thm:convfullydiscrete} with $v_{h,\tau}=P^\tau P_h u$ and
 $\psi_{h,\tau}=\Pi^\tau \Pi_h \phi$.
\end{proof}
\begin{remark}
 In~\cref{cor:convratefullydisc} it is enough to demand $\phi\in H^r(0,T; H^{-1/2+s}(\Er))$ and
 $g_2\in L^2(0,T; H^{-1/2+s}(\Er))$ if $s>1/2$, see \cref{rem:BEMregular}.
 The constraint $\tau\leq 1/4$ in~\cite[Lemma 14, Theorems 15, 24]{Egger:2017} is needed there since
 the bilinear form only satisfies a G\r{a}rding inequality.
\end{remark}

\subsection{The classical backward Euler scheme}
\label{sec:classicalEuler}
In the following we define a classical backward Euler approach for the time discretization
of the semi-discrete system~\cref{eq1:fvmbem}--\cref{eq2:fvmbem} or \cref{eq:FVMBEMadded}.
In contrast to~\cref{sec:variantEuler} we require more regularity in the time component 
for some \emph{model input data}, namely, $q\in L^2(\Omega)$,
$f\in H^1(0,T;L^2(\Omega))$, $g_1\in H^1(0,T;H^{1/2}(\Gamma))$, 
and $g_2\in H^1(0,T;L^2(\Gamma))$.
With the notation introduced in the beginning of \cref{sec:fully-disc}
the fully discrete system reads:
\begin{problem}[ClaBE-FVM-BEM]
Set $u_h^0=P_h q\in \S^1(\T)$. Find sequences $u_h^n \subset \S^1(\T)$ and
$\phi_h^n \subset \P^0(\Er)$ 
for $n=1,\ldots,N$ such that 
 \begin{align}
  \label{eq:vp1htauclassic}
 \product{\dtau u_{h}^n}{\IIh^* v_h}{\Omega} 
 + \AA_V(\nabla u_{h}^n,\IIh^* v_h) 
 - \product{\phi_{h}^n}{\IIh^* v_h}{\Gamma} 
 &= \product{f^n}{\IIh^* v_h}{\Omega}+
    \product{g^n_2}{\IIh^* v_h}{\Gamma},\\
 \label{eq:vp2htauclassic}
 \product{(1/2-\K) u_{h}^n} {\psi_h}{\Gamma} 
 + \product{\V\phi_{h}^n}{\psi_h}{\Gamma} 
 &= \product{(1/2-\K)g^n_1}{\psi_h}{\Gamma} 
 \end{align}
for all $v_h\in\S^1(\T)\subset H^1(\Omega)$ and $\psi_h\in\P^0(\T)\subset H^{-1/2}(\Gamma)$.\\
In compact notation: Find the sequence $\u_h^n=(u_h^n,\phi_h^n)\in \HH^h=\S^1(\T)\times \P^0(\Er)$
for $n=1,\ldots,N$ with $u_h^0=P_h q$ such that
\begin{align}
\label{eq:FVMBEMfullydisc}
 \product{ \dtau u_{h}^n}{\IIh^* v_h}{\Omega}&+\BB_V(\u_h^{n};\v_h) 
  =F_V(\v_h,t^n)
\end{align}
for all $\v_h=(v_h,\psi_h)\in\HH^h= \S^1(\T) \times \P^0(\Er)$,
where $F_V(\v_h,t^n)$ is defined in~\cref{eq:Ffvm}.
\end{problem}
\begin{remark}
 In fact, the system ClaBE-FVM-BEM~\cref{eq:vp1htauclassic}--\cref{{eq:vp2htauclassic}} 
 only differs
 from the variant
 VarBE-FVM-BEM~\cref{eq:vp1htauweight}--\cref{eq:vp2htauweight} in the right-hand side,
 see also \cref{rem:classicalEuler}.
\end{remark}
To analyze the system~\cref{eq:FVMBEMfullydisc} 
we frequently use a Taylor series approximation
in the time component of the following type.
\begin{lemma}
 Let $g\in H^1([0,T])$. Then
 \begin{align}
  \label{eq:taylor}
   \sum_{n=1}^N \tau^n(g^n)^2=\sum_{n=1}^N \tau^n g(t^n)^2
  \lesssim \norm{g}{L^2(0,T)}^2+\tau^2\norm{g'}{L^2(0,T)}.
 \end{align}
\end{lemma}
\begin{proof}
 For $t^{n-1}\leq t\leq t^n$ we see with Taylor expansion and the Cauchy-Schwarz inequality
 \begin{align*}
  g(t^n)^2\leq 2\Big[g^2(t)+\Big(\int_{t^{n-1}}^t g'(s)\,ds\Big)^2\Big]
  \leq 2\Big[g^2(t)+\tau^n\int_{t^{n-1}}^{t^n}(g'(s))^2\,ds\Big]
 \end{align*}
 Integration over $[t^{n-1},t^n]$ and summing over $n=1,\ldots,N$ leads to the assertion.
\end{proof}

We consider the solutions $u_h^n$ and $\phi_h^n$ of~\cref{eq:FVMBEMfullydisc}
to be approximations for $u(t^n)$ and $\phi(t^n)$, respectively.
First we state the unique solvability of our fully discrete system:
\begin{lemma}[Well-posedness and discrete energy estimate]
 \label{lm:discreteWellp}
The solution $(u_h^n,\phi_h^n)\in \HH^h=\S^1(\T)\times \P^0(\Er)$ for $n=1,\ldots,N$
of \cref{eq:FVMBEMfullydisc} is unique and fulfills 
\begin{align*}
 &\sum_{n=1}^N \tau^n \big(\norm{u_h^n}{H^1(\Omega}^2+\norm{\phi_h^n}{H^{-1/2}(\Gamma)}^2\big) \\
 &\quad\leq C \big( \norm{q}{L^2(\Omega)}^2 + \norm{f}{H^1(0,T;L^2(\Omega))}^2 
 + \norm{g_1}{H^1(0,T;H^{1/2}(\Gamma))}^2 + \norm{g_2}{H^1(0,T;L^2(\Gamma))}^2\big).
\end{align*} 
\end{lemma}
\begin{proof}
 Testing~\cref{eq:FVMBEMfullydisc} with $\v_h=(v_h,\psi_h)=(u_h^n,\phi_h^n)$, \cref{eq:dtdiscrete},
 the ellipticity of $\BB_V$, $\norm{\IIh^* u_h^n}{L^2(\Omega)}\leq C\norm{u_h^n}{L^2(\Omega)}$,
 and standard arguments lead to
 \begin{align}
  \begin{split}
   \label{eq:discenergy1}
  &\norm{u_h^{N}}{L^2(\Omega)}^2 + \sum_{n=1}^N \tau^n\left(\norm{u_h^{n}}{H^1(\Omega)}^2
  + \norm{\phi_h^{n}}{H^{-1/2}(\Gamma)}^2 \right) \\
  &\quad\leq C \Big( \norm{q}{L^2(\Omega)}^2 + \sum_{n=1}^N \tau^n\big( \norm{f^{n}}{L^2(\Omega)}^2 
  + \norm{g_2^{n}}{L^2(\Gamma)}^2 +\norm{g_1^{n}}{H^{1/2}(\Gamma)}^2\big) \Big).
  \end{split}
 \end{align}
 Due to the regularity of the  model data we may apply~\cref{eq:taylor} to show the assertion.
\end{proof}
The following theorem provides the convergence of the fully discrete scheme. 
\begin{theorem}[Convergence of the fully-discrete discrete system ClaBE-FVM-BEM]
\label{thm:clafulldisc}
There exists $h_{\max}>0$ such that for $h$ sufficiently small, i.e., $h<h_{\max}$ the following statement holds:
Let $\lambda_{\min}(\A) - \frac{1}{4} C_{\K} > 0$, $C_{\K}\in [1/2,1)$.
Moreover, let $u$ and $\phi$ and the data be sufficiently smooth. 
Then the solution $\u_h^n=(u_h^n,\phi_h^n)\in\HH^h=\S^1(\T)\times \P^0(\Er)$ 
of~\cref{eq:FVMBEMfullydisc}
converges to the weak solution $\u=(u,\phi)$ of~\cref{eq:weakAddedForm}. More precisely:
if $u\in  H^1(0,T;H^1(\Omega))$, $\partial_{t} u\in H^1(0,T;H^1(\Omega)')$, 
$\partial_{tt} u\in L^2(0,T;H^1(\Omega)')$,
and $\phi\in  H^1(0,T;H^{-1/2}(\Gamma))$, and the data model $q\in L^2(\Omega)$, 
$f\in H^1(0,T;L^2(\Omega))$, $g_1\in H^1(0,T;H^{1/2}(\Gamma))$, and
$g_2\in H^1(0,T;L^2(\Gamma))$ there holds
\begin{align*}
  &\Big[\sum_{n=1}^N\tau^n\norm{\u(t^n)-\u_h^n}{\HH}^2\Big]^{1/2}
  =\Big[\sum_{n=1}^N\tau^n\Big(\norm{u(t^n)-u_h^n}{H^1(\Omega)}^2 
  +  \norm{\phi(t^n)-\phi_h^n}{H^{-1/2}(\Gamma)}^2 \Big)\Big]^{1/2} \\
  &\leq C\Big[
  \norm{\u-\v_h}{\HH_T}+\norm{\dt u-\dt v_h}{H_T'}
  + h^{1/2}\Big(\norm{g_2-\overline{g}_2}{L^2_{T,\Gamma}}
    +\tau\norm{\dt g_2-\dt \overline{g}_2}{L^2_{T,\Gamma}}\Big)\\
  &\quad + h\Big(\norm{\dt v_h}{L^2_{T,\Omega}}
  +\norm{f}{L^2_{T,\Omega}}+\tau\norm{\dt f}{L^2_{T,\Omega}}
  +\norm{v_h}{H_T}+\tau\norm{\dt v_h}{H_T}\Big)\\
  &\quad+\tau\Big(\norm{\partial_{tt} u}{H_T'}+\norm{\dt \u-\dt \v_h}{\HH_T}
  \Big)
  \Big]
\end{align*}
for all $\v_h=(v_h,\psi_h)\in H^1(0,T;\S^1(\T))\times H^1(0,T;\P^0(\Er))$
with $v_h(0)=P_h q$.
The statement also holds if we use the upwind stabilized bilinear form with $\AA_V^{up}$.
\end{theorem}
\begin{proof}
 First we split the error into an approximation error and a discrete error component,
 i.e., for arbitrary 
 $\v_h=(v_h,\psi_h)\in H^1(0,T;\S^1(\T)\times H^1(0,T;\P^0(\Er))$ with
 $v_h(0)=P_h q$ we split
 \begin{align}
  \label{eq:clahelp1}
  \norm{\u(t^n)-\u_h^n}{\HH} 
  &\leq  \norm{\u(t^n)-\v_h^n}{\HH} + \norm{\u_h^n-\v_h^n}{\HH}
 \end{align}
 for $n=1,\ldots, N$.
 Next we estimate the discrete error part
 and define $\w_h^n:=(w_h^n,\varphi_h^n):=\u_h^n-\v_h^n\in\HH^h$. Note that $w_h^0=0$.
 Following exactly the lines of the proof for~\cref{thm:convfullydiscrete}
 but using~\cref{eq:weakAddedForm} evaluated in $t^n$ and~\cref{eq:FVMBEMfullydisc} we arrive at
 \begin{align}
  \label{eq:clahelp2}
  \begin{split}
  \frac{1}{2\tau^n}\big(\norm{w^n_{h}}{\chi}^2 &- \norm{w^{n-1}_{h}}{\chi}^2\big)
  +\norm{\w^n_{h,\tau}}{\HH}^2\\
  &\lesssim  \dual{\dt u(t^n)}{w^n_h}{\Omega}
  - \product{\dtau v^n_h}{\IIh^*w^n_h}{\Omega}
  + F_V(\w^n_h;t^n) - F(\w^n_h;t^n) \\
  &\quad+ \BB(\v^n_h;\w^n_h)
  - \BB_V(\v^n_h;\w^n_h)
  + \BB(\u(t^n) - \v^n_h;\w^n_h).
  \end{split}
 \end{align}
For the difference of the first two terms on the right-hand side we see with
the Cauchy-Schwarz inequality, and the estimate~\cref{eq:piecewisepropT}
 that
\begin{align*}
&\dual{\dt u(t^n)}{w^n_h}{\Omega}
  - \product{\dtau v^n_h}{\IIh^*w^n_h}{\Omega} \\
&=\dual{\partial_t u(t^n)-\dtau u(t^n)}{w^n_h}{\Omega} 
+ \dual{\dtau u(t^n)-\dtau v_h^n}{w^n_h}{\Omega} 
+ \product{\dtau v_h^n}{w^n_h-\IIh^* w^n_h}{\Omega}\\
&\lesssim \Big( \norm{\partial_t u(t^n)-\dtau u(t^n)}{H^1(\Omega)'}
+\norm{\dtau u(t^n)-\dtau v^n_h}{H^1(\Omega)'}
+h\norm{\dtau v^n_h}{L^2(\Omega)}\Big) \norm{w^n_h}{H^1(\Omega)}.
\end{align*}
The other terms in~\cref{eq:clahelp2} can be bounded as before, 
using~\cref{lem:rhsdifference},~\cref{lem:bildifference},
and the continuity of the bilinear form $\BB$. 
With standard manipulations we estimate
 \begin{align*}
   &\frac{1}{2}\norm{w^N_{h,\tau}}{\chi}^2
   +\sum_{n=1}^N\tau^n\norm{\w^n_{h,\tau}}{\HH}^2\\
   &\quad\lesssim \sum_{n=1}^N\tau^n
   \Big(\norm{\partial_t u(t^n)-\dtau u(t^n)}{H^1(\Omega)'}^2
+\norm{\dtau u(t^n) - \dtau v^n_h}{H^1(\Omega)'}^2
+h^2\norm{\dtau v^n_h}{L^2(\Omega)}^2\\
   &\qquad\qquad+h^2\norm{f^n}{L^2(\Omega)}^2
   + h\norm{g_2^n-\overline{g}_2(t^n)}{L^2(\Gamma)}^2 
    + h^2\norm{v^n_h}{H^1(\Omega)}^2
   +\norm{\u(t^n)-\v^n_h}{\HH}^2 \Big).
 \end{align*}
 With classical Taylor series, i.e., with the integral form of the remainder,
 we estimate 
 \begin{align*}
  \sum_{n=1}^N\tau^n\norm{\partial_t u(t^n)-\dtau u(t^n)}{H^1(\Omega)'}^2  &\leq
  \tau^2\norm{\partial_{tt}u}{H_T'}^2,\\
  \sum_{n=1}^N\tau^n\norm{\dtau u(t^n) - \dtau v^n_h}{H^1(\Omega)'}^2  &\leq
  \norm{\dt u-\dt v_h}{H_T'}^2,\\
  \sum_{n=1}^N\tau^n\norm{\dtau v^n_h}{L^2(\Omega)}^2  &\leq
  \norm{\dt v_h}{L^2_{T,\Omega}}^2.
 \end{align*}
 For all the other terms we use~\cref{eq:taylor} to finally prove the assertion 
 with the error splitting~\cref{eq:clahelp1}.
\end{proof} 
For simplicity we only state first order convergence which follows
directly from \cref{thm:clafulldisc} with
the aid of \cref{lem:approx}.
\begin{corollary}[First order convergence of the fully-discrete ClaBE-FVM-BEM]
 \label{cor:clafirstorder}
Let $P_h$ be $H^1$-stable, e.g., $\T$ is quasi-uniform.
Additionally to the assumptions of~\cref{thm:clafulldisc} we require
$u\in  H^1(0,T;H^2(\Omega))$, $\partial_{t} u\in H^1(0,T;L^2(\Omega))$, 
$\partial_{tt} u\in L^2(0,T;H^1(\Omega)')$,
and $\phi\in  H^1(0,T;H^{1/2}(\Gamma))$, and for the model input data $q\in L^2(\Omega)$,
$f\in H^1(0,T;L^2(\Omega))$, $g_1\in H^1(0,T;H^{1/2}(\Gamma))$, and
$g_2\in H^1(0,T;H^{1/2}(\Gamma))$ there holds

\begin{align*}
  &\Big[\sum_{n=1}^N\tau^n\Big(\norm{u(t^n)-u_h^n}{H^1(\Omega)}^2 
  +  \norm{\phi(t^n)-\phi_h^n}{H^{-1/2}(\Gamma)}^2 \Big)\Big]^{1/2}=\O(\tau+h).
\end{align*}
This also holds if we use the upwind stabilized bilinear form $\AA_V^{up}$
instead of $\AA_V$.
\end{corollary}

\begin{remark}
 The left-hand sides from~\cref{thm:clafulldisc} and~\cref{cor:clafirstorder}
 are discrete versions of the norm $\norm{\cdot}{\HH_T}$. By some linear interpolation
 and with \cref{eq:winequality1} we state the assertions as 
 for the version with the variant backward Euler time discretization scheme in \cref{sec:variantEuler}.
 In~\cref{cor:clafirstorder} it is enough to demand $\phi\in H^1(0,T; H^{1/2}(\Er))$ and
 $g_2\in H^1(0,T; H^{1/2}(\Er))$ if $s>1/2$, see \cref{rem:BEMregular}. 
 The analysis with a Crank-Nicolson time discretization follows easily from the theory developed
 in this section.
\end{remark}

\section{Numerical illustration}
\label{sec:numerics}
To illustrate the theoretical findings we will present three examples in two dimensions in
this section.
The calculations have been performed with \textsc{Matlab} using some functions
from the \textsc{Hilbert}-package~\cite{HILBERT:2013-1} for the matrices resulting
from the integral operators $\V$ and $\K$.
Because the norm $\norm{\phi(t)-\phi_h(t)}{H^{-1/2}(\Gamma)}$ is not computable,
we will use the equivalent norm
\begin{align*}
 \norm{\phi(t)-\phi_h(t)}{H^{-1/2}(\Gamma)} \sim
 \norm{\phi(t)-\phi_h(t)}{\V} := \dual{\V(\phi(t)-\phi_h(t))}{\phi(t)-\phi_h(t)}{\Gamma},
\end{align*}
see~\cite{Erath:2010-phd} for details.
Hence $\norm{\phi-\phi_h^n}{L^2(0,T;\V)}$ is an equivalent norm to
$\norm{\phi-\phi_h^n}{L^2(0,T;H^{-1/2}(\Gamma))}$.
All other spatial norms and time integrals are approximated by Gaussian quadrature.
We present results with the variant backward Euler time discretization scheme.
Note that in practice we implement~\crefrange{eq:vp1htau}{eq:vp2htau} 
instead of~\crefrange{eq:vp1htauweight}{eq:vp2htauweight}.
In all examples we divide $\Omega$ into congruent triangles with a mesh size $h=0.125$.
We divide the time interval $[0,1]$ into uniform time steps with step size $\tau^n=\tau=0.05$.
The refinement will be uniform for both, the space and the time grid, simultaneously.
%
\subsection{Convection dominated diffusion-convection-reaction problem} \label{subsec:bsp1}
The first example has a prescribed smooth analytical solution.
In the domain $\Omega=(0,1/2)^2$,
we choose
\begin{align*}
u(x_1,x_2)=0.5(1+t)\left( 1-\tanh\left(\frac{0.25-x_1}{0.02}\right)\right),
\end{align*}
and as the solution in the corresponding exterior domain $\Omega_e$
\begin{align*}
u_e(x_1,x_2)=(1-t)\log \sqrt{(x_1-0.25)^2+(x_2-0.25)^2}.
\end{align*}
The interior solution has a simulated shock in the middle of the domain, which
can pose certain difficulties to the used method.

The diffusion $\A=\alpha\mathbf{I}$ has a jump, i.e.,
\begin{align*}
 \alpha=\begin{cases} 0.42\quad &\text{for}~x_2 < 0.25, \\
1 \quad &\text{for}~x_2 \geq 0.25. \end{cases}
\end{align*}
The convection field and the reaction coefficient are
set to $\b=(1000x_1,0)^T$ and $\c=5$, respectively.
Furthermore, the jumps $g_1$, $g_2$, and the right-hand side $f$ are calculated by means of
the analytical solution.
Because the problem is convection-dominated we use the full upwind stabilization
$\AA_V^{up}$ defined in~\cref{eq:fvmupwind}.
Both the interior and the exterior solution are smooth, thus we
expect first order convergence as predicted by~\cref{cor:convratefullydisc}.
This can be seen in~\cref{fig:errorbsp1}.

\begin{figure}
\centering
\begin{tikzpicture}
\begin{loglogaxis}[width=0.75\textwidth,
xlabel={$1/h$}, ylabel={\small error}, font={\scriptsize},
legend style={font=\small, draw=none, fill=none, cells={anchor=west}, legend pos=south west}]
\addplot table [x=hinv,y=V_error] {figures/tanh_var.dat};
\addplot table [x=hinv,y=H1_error] {figures/tanh_var.dat};
\addplot table [x=hinv,y=Energy_error] {figures/tanh_var.dat};

\logLogSlopeTriangle{0.87}{0.17}{0.56}{-1}{black}{\scriptsize};

\legend{
$\norm{\phi - \phi_{h,\tau}}{L^2(0,T;\V)}$,
$\norm{u-u_{h,\tau}}{H_T}$,
$\norm{u-u_{h,\tau}}{H_T}+\norm{\phi - \phi_{h,\tau}}{L^2(0,T;\V)}$}
\end{loglogaxis}
\end{tikzpicture}
\caption{The different error components of the solutions
$u_{h,\tau}$ and $\phi_{h,\tau}$
for uniform refinement in time and space
for the smooth example in~\cref{subsec:bsp1}.
The added energy error norm
$(\norm{u-u_{h,\tau}}{H_T}^2+\norm{\phi - \phi_{h,\tau}}{L^2(0,T;\V)}$
shows first order convergence.}
\label{fig:errorbsp1}
\end{figure}
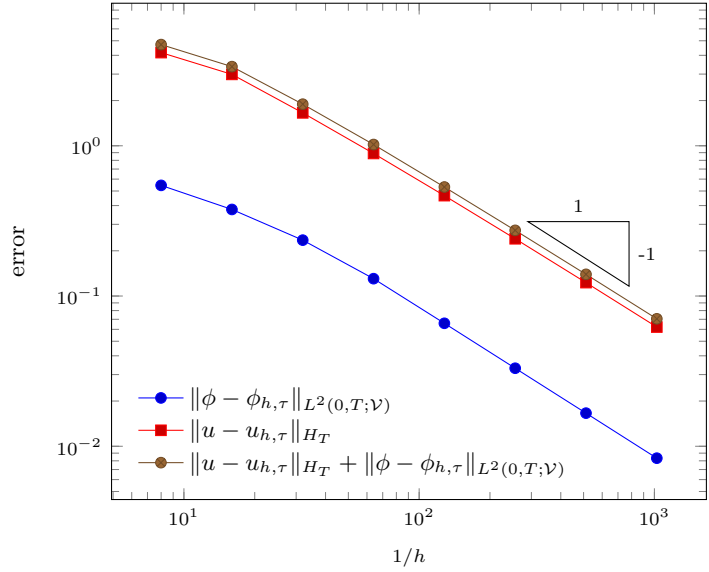

%
%
\subsection{Diffusion problem on an L-shaped domain}
 \label{subsec:bsp2}
 The second test shows the reduction of the order of convergence if
 we do not meet the regularity requirements.
 We consider a purely diffusive problem of model problem~\crefrange{eq:model1}{eq:model6}
 i.e., $\b=(0,0)^T$ and $\c=0$.
 The diffusion matrix is chosen as
 \begin{align*}
 \A=  \left ( \begin{array}{rr}
  10+\cos x_1 & 160\,x_1 x_2 \\
  160\,x_1 x_2 & 10+\sin x_2
  \end{array}\right).
\end{align*}
On the L-shaped domain $\Omega= (-1/4, 1/4)^2 \setminus [0, 1/4]\times [-1/4,0]$
we prescribe a function with a singularity in the corner $(0,0)$:
Let $x=(x_1,x_2) = r(\cos\varphi, \sin\varphi)$ with $r\in\mathbb{R}^+$ and $\varphi\in [0,2\pi)$
be the polar coordinates of a point $x$, then the analytical solution in the interior reads:
\begin{align*}
u(x_1,x_2) = (1+t^2) r^{2/3} \sin(2\varphi/3).
\end{align*}
In the exterior domain $\Omega_e$ we choose
\begin{align*}
u_e(x_1,x_2)=(1-t)\log \sqrt{(x_1+0.125)^2+(x_2-0.125)^2}.
\end{align*}
As above we compute the jumps $g_1$, $g_2$, and the right-hand side $f$ accordingly.
The function in the interior has reduced regularity in space and 
it is only in $H^{1+2/3-\varepsilon}(\Omega)$ for every $\varepsilon>0$.
Hence,~\cref{cor:convratefullydisc} predicts a reduced convergence order of
$\O(h^{2/3})$, which is indeed observed in the convergence plot~\cref{fig:errorbsp2} 
of our numerical approximation.

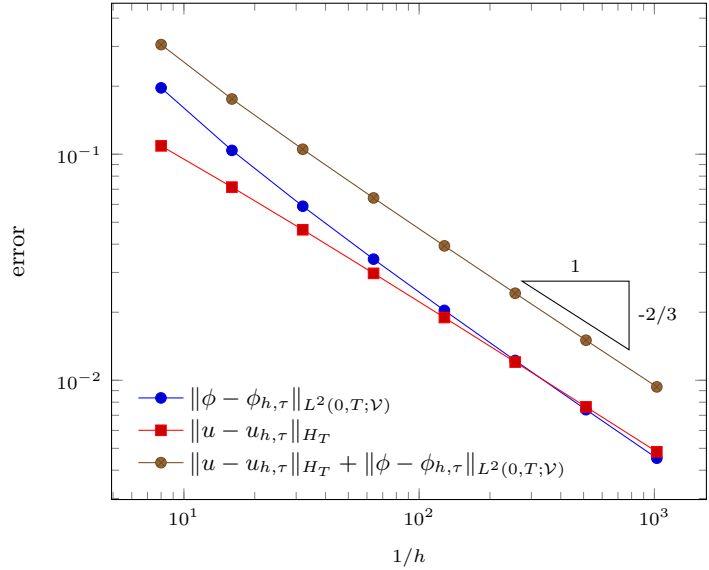
\begin{figure}
\centering
\begin{tikzpicture}
\begin{loglogaxis}[width=0.75\textwidth,
xlabel={$1/h$}, ylabel={\small error}, font={\scriptsize},
legend style={font=\small, draw=none, fill=none, cells={anchor=west}, legend pos=south west}]
\addplot table [x=hinv,y=V_error] {figures/sinusLaplaceMatrix_var.dat};
\addplot table [x=hinv,y=H1_error] {figures/sinusLaplaceMatrix_var.dat};
\addplot table [x=hinv,y=Energy_error] {figures/sinusLaplaceMatrix_var.dat};

\logLogSlopeTriangle{0.87}{0.18}{0.44}{-2/3}{black}{\scriptsize};

\legend{
$\norm{\phi - \phi_{h,\tau}}{L^2(0,T;\V)}$,
$\norm{u-u_{h,\tau}}{H_T}$,
$\norm{u-u_{h,\tau}}{H_T}+\norm{\phi - \phi_{h,\tau}}{L^2(0,T;\V)}$}
\end{loglogaxis}
\end{tikzpicture}
\caption{The different error components of the solutions
$u_{h,\tau}$ and $\phi_{h,\tau}$
for uniform refinement in time and space
for the non-smooth example in space in~\cref{subsec:bsp2}.
The added energy error norm
$(\norm{u-u_{h,\tau}}{H_T}^2+\norm{\phi - \phi_{h,\tau}}{L^2(0,T;\V)}$
shows a reduced order of convergence.}
\label{fig:errorbsp2}
\end{figure}

%
%
\subsection{A more practical problem}
 \label{subsec:bsp3}

The last example is a more practical example, where we do not know the
analytical solution. Let $\Omega=(-1/4,1/4)^2$. The diffusion $\A=\alpha\mathbf{I}$
is set to
\begin{align*}
 \alpha=\begin{cases} 10^{-2}\quad &\text{for}~x_1 < 0.25, \\
        10^{-3}\quad &\text{for}~x_1 \geq 0.25, \end{cases}
\end{align*}
the convection to $\b=(0.25-4x_2, 4x_1)^T$, and the reaction to $\c=1$.
The jumps are chosen to be zero and the right-hand side is chosen as
\begin{align*}
  f(x_1,x_2, t)=\begin{cases} 50 \quad &\text{for}~-0.2\leq x_1\leq -0.1,\quad -0.2\leq x_2\leq -0.05, \quad t<0.25, \\
     25 \quad &\text{for}~-0.2\leq x_1\leq -0.1,\quad 0.05\leq x_2\leq 0.2, \quad t<0.5, \\
     0 \quad &\text{else.} \end{cases}
\end{align*}
This right-hand side may simulate a chemical compound being injected in two areas
until a certain point in time ($t=0.25$ and $t=0.5$). Hence, our model problem describes 
the transport
of this compound in a (porous) medium. Due to convection dominance we apply
the full upwind stabilization $\AA_V^{up}$ defined in~\cref{eq:fvmupwind}.
The solution is plotted at different times in~\cref{fig:bps3}.

\begin{figure}[tbhp]
\centering
\captionsetup[subfigure]{width=100pt}
\subfloat[Solution at $t=0.0625$.]{\includegraphics[width=.3\textwidth]{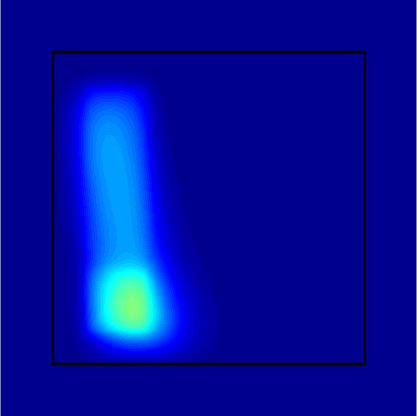}}
\hspace{.025\textwidth}
\subfloat[Solution at $t=0.125$.]{\includegraphics[width=.3\textwidth]{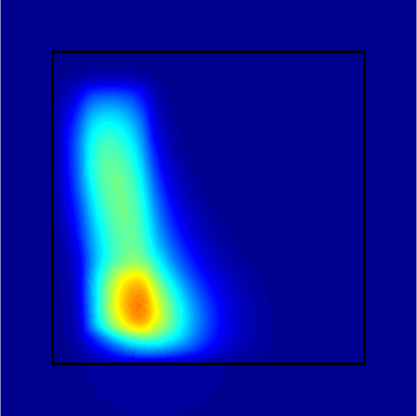}}
\hspace{.025\textwidth}
\subfloat[Solution at $t=0.25$.]{\includegraphics[width=.3\textwidth]{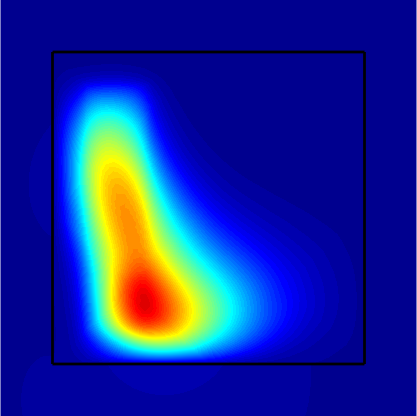}}
\vspace{.1\baselineskip}\hspace{.025\textwidth}
\subfloat[Solution at $t=0.5$.]{\includegraphics[width=.3\textwidth]{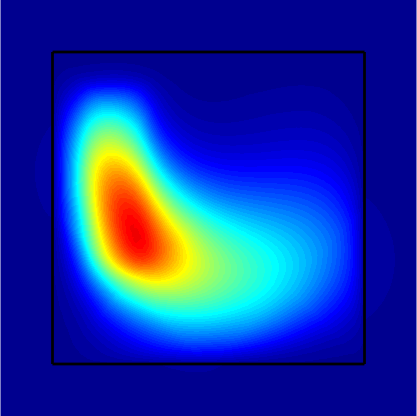}}
\hspace{.025\textwidth}
\subfloat[Solution at $t=0.75$.]{\includegraphics[width=.3\textwidth]{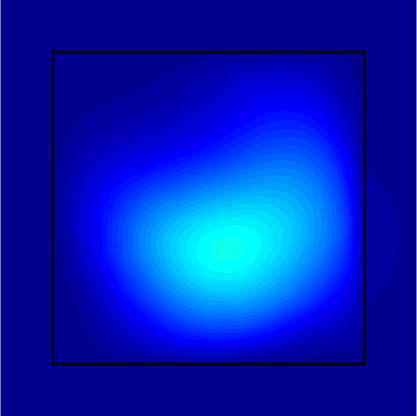}}
\hspace{.025\textwidth}
\subfloat[Solution at $t=1$.]{\includegraphics[width=.3\textwidth]{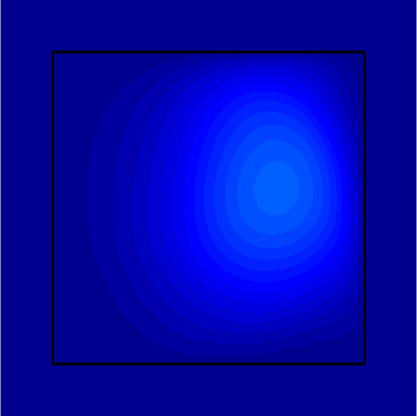}}
\vspace{.25\baselineskip}
\includegraphics[width=0.5\textwidth,clip]{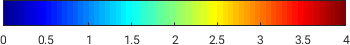}
\caption{Solution of the transport problem in \cref{subsec:bsp3} at different times.
The source is located at the left-hand side of the domain and stronger in the lower half.
We turn of the source in the lower half at $t=0.25$. At $t=0.5$ the source is turned off
completely.}
\label{fig:bps3}
\end{figure}

\section{Conclusions}\label{sec:conclusions}
In this paper we considered parabolic-elliptic problems, where
the interior problem can be convection dominated and the exterior domain
is unbounded. The coupling of the finite volume method (for the interior problem)
and the boundary element method (for the exterior problem) has been proven to
be a good choice for spatial discretization
to handle all difficulties arising from this kind of interface problems.
We showed that the semi-discrete FVM-BEM coupling yields to unique and stable solutions
and converges under minimal regularity assumptions. The subsequent discretization in time by
a variant of the backward Euler method yields to a fully discrete scheme that also converges 
under minimal regularity assumptions on the solution.
As an alternative we provided a time discretization with a classical backward Euler scheme under
standard regularity assumptions on the solution in the time component.
Note that our analysis can also be applied for standalone FVM approximation (replace coupling conditions
by boundary conditions) which improves available results in 
the literature.

\bibliographystyle{amsalpha}
\bibliography{literature} 

\newcommand{\etalchar}[1]{$^{#1}$}
\providecommand{\bysame}{\leavevmode\hbox to3em{\hrulefill}\thinspace}
\providecommand{\MR}{\relax\ifhmode\unskip\space\fi MR }
\providecommand{\MRhref}[2]{%
  \href{http://www.ams.org/mathscinet-getitem?mr=#1}{#2}
}
\providecommand{\href}[2]{#2}
\begin{thebibliography}{ACF{\etalchar{+}}11}

\bibitem[ACF{\etalchar{+}}11]{Augustin:2001}
M.~Augustin, A.~Caiazzo, A.~Fiebach, J.~Fuhrmann, V.~John, A.~Linke, and
  R.~Umla, \emph{An assessment of discretizations for convection-dominated
  convection-diffusion equations}, Computer Methods in Applied Mechanics and
  Engineering \textbf{200} (2011), no.~47, 3395 -- 3409.

\bibitem[AEF{\etalchar{+}}14]{HILBERT:2013-1}
M.~Aurada, M.~Ebner, M.~Feischl, S.~Ferraz-Leite, T.~F{\"u}hrer, P.~Goldenits,
  M.~Karkulik, M.~Mayr, and D.~Praetorius, \emph{{HILBERT} --- a {MATLAB}
  implementation of adaptive {2D-BEM}}, Numer. Algor. \textbf{67} (2014),
  1--32.

\bibitem[BR87]{Bank:1987}
R.E. Bank and D.J. Rose, \emph{Some error estimates for the box method}, SIAM
  J. Numer. Anal \textbf{24} (1987), 777--787.

\bibitem[Cia78]{Ciarlet:1978-book}
P.~G. Ciarlet, \emph{The finite element method for elliptic problems}, Studies
  in mathematics and its applications, North-Holland, Amsterdam, New-York,
  1978.

\bibitem[CL99]{Chou:1999}
S.-H. Chou and Q.~Li, \emph{Error estimates in {$L^2$}, {$H^1$}, {$L^\infty$}
  in covolume methods for elliptic and parabolic problems: a unified approach},
  Mathematics of Computation \textbf{69} (1999), no.~229, 103--120.

\bibitem[CLT04]{Chatzipantelidis:2004}
P.~Chatzipantelidis, R.~D. Lazarov, and V.~Thom{\'e}e, \emph{Error estimates
  for a finite volume element method for parabolic equations in convex
  polygonal domains}, Numerical Methods for Partial Differential Equations
  \textbf{20} (2004), no.~5, 650--674.

\bibitem[Cos88]{Costabel:1988-1}
M.~Costabel, \emph{Boundary integral operators on {L}ipschitz domains:
  elementary results}, SIAM J. Math. Anal. \textbf{19} (1988), 613--626.

\bibitem[EES17]{Egger:2017}
H.~Egger, C.~Erath, and R.~Schorr, \emph{On the non-symmetric coupling method
  for parabolic-elliptic interface problems}, Preprint, arXiv:1711.08487
  (2017), 1--24.

\bibitem[ELL02]{Ewing:2002}
R.E. Ewing, T.~Lin, and Y.~Lin, \emph{On the accuracy of the finite volume
  method based on piecewise linear polynomials}, SIAM J. Numer. Anal
  \textbf{39} (2002), 1865--1888.

\bibitem[EOS17]{Erath:2017-1}
C.~Erath, G.~Of, and F.-J. Sayas, \emph{A non-symmetric coupling of the finite
  volume method and the boundary element method}, Numer. Math. \textbf{135}
  (2017), 895--922.

\bibitem[EP16]{Erath:2016-1}
C.~Erath and D.~Praetorius, \emph{{Adaptive vertex-centered finite volume
  methods with convergence rates}}, SIAM J. Numer. Anal. \textbf{54} (2016),
  no.~4, 2228--2255.

\bibitem[EP17]{Erath:2017-2}
\bysame, \emph{C{\'e}a-type quasi-optimality and convergence rates for
  (adaptive) vertex-centered fvm}, Finite Volumes for Complex Applications VIII
  - Methods and Theoretical Aspects, Springer International Publishing, 2017,
  pp.~215--223.

\bibitem[Era10]{Erath:2010-phd}
C.~Erath, \emph{{Coupling of the Finite Volume Method and the Boundary Element
  Method - Theory, Analysis, and Numerics}}, Ph.D. thesis, University of Ulm,
  2010.

\bibitem[Era12]{Erath:2012-1}
\bysame, \emph{{Coupling of the finite volume element method and the boundary
  element method: an a priori convergence result}}, SIAM J. Numer. Anal.
  \textbf{50} (2012), no.~2, 574--594.

\bibitem[Era13]{Erath:2013-2}
\bysame, \emph{{A new conservative numerical scheme for flow problems on
  unstructured grids and unbounded domains}}, J. Comput. Phys. \textbf{245}
  (2013), 476--492.

\bibitem[Eva10]{Evans:2010-book}
L.~C. Evans, \emph{Partial differential equations}, Graduate studies in
  mathematics, American Mathematical Society, 2010.

\bibitem[Hac89]{Hackbusch:1989}
W.~Hackbusch, \emph{On first and second order box schemes}, Computing
  \textbf{41} (1989), no.~4, 277--296.

\bibitem[JN80]{Johnson:1980-1}
C.~Johnson and J.~C. N{\'e}d{\'e}lec, \emph{{On the coupling of boundary
  integral and finite element methods}}, Math. Comput. \textbf{35} (1980),
  1063--1079.

\bibitem[McL00]{McLean:2000-book}
W.~McLean, \emph{Strongly elliptic systems and boundary inte gral equations},
  Cambridge University Press, 2000.

\bibitem[MS87]{MacCamy:1987}
R.~C. MacCamy and M.~Suri, \emph{A time-dependent interface problem for
  two-dimensional eddy currents}, Quart. Appl. Math. \textbf{44} (1987),
  675--690.

\bibitem[RST08]{Roos:2008-book}
H.~G. Roos, M.~Stynes, and L.~Tobiska, \emph{{Numerical methods for singularly
  perturbed differential equations}}, second ed., Springer, Berlin, Berlin,
  Heidelberg, 2008.

\bibitem[Tan14]{Tantardini:2014-1}
F.~Tantardini, \emph{Quasi-optimality in the backward euler-galerkin method for
  linear parabolic problems}, Ph.D. thesis, Universit{\`a} degli Studi di
  Milano, Milan, Italy, 2014.

\end{thebibliography}

\end{document}